\documentclass[12pt]{amsart}
\usepackage{a4wide,enumerate,xcolor}
\usepackage{amsmath,graphicx,comment}
\allowdisplaybreaks

\usepackage{enumitem}
\setlist[itemize]{label={$\bullet$}, leftmargin=12pt, itemsep=3pt}

\let\pa\partial
\let\na\nabla
\let\eps\varepsilon
\newcommand{\N}{{\mathbb N}}
\newcommand{\R}{{\mathbb R}}
\newcommand{\dom}{\Omega}
\newcommand{\diver}{\operatorname{div}}
\newcommand{\D}{{\mathbb D}} 
\newcommand{\K}{{\mathbb K}}
\renewcommand{\S}{\mathbb{S}}
\newcommand{\T}{{\mathbb T}}
\newcommand{\tr}{\text{tr}}
\DeclareMathOperator{\supp}{supp}

\newtheorem{theorem}{Theorem}
\newtheorem{lemma}[theorem]{Lemma}

\newtheorem{remark}[theorem]{Remark}

\newtheorem{definition}{Definition}

\begin{document}

\title[Compressible Navier--Stokes--Korteweg equations]{Global existence analysis for a class of \\ compressible Navier--Stokes--Korteweg equations}

\author[A. J\"ungel]{Ansgar J\"ungel}
\address{Institute of Analysis and Scientific Computing, TU Wien, Wiedner Hauptstra\ss e 8--10, 1040 Wien, Austria}
\email{juengel@tuwien.ac.at} 

\author[F. Philipp]{Flora Philipp}
\address{Institute of Analysis and Scientific Computing, TU Wien, Wiedner Hauptstra\ss e 8--10, 1040 Wien, Austria}
\email{flora.philipp@tuwien.ac.at} 

\date{\today}

\thanks{The authors acknowledge partial support from the Austrian Science Fund (FWF), grant 10.55776/PAT2687825, and from the Austrian Federal Ministry for Women, Science and Research and implemented by \"OAD, grant MULT09/2025. This work has received funding from the European Research Council (ERC) under the European Union's Horizon 2020 research and innovation programme, ERC Advanced Grant NEUROMORPH, no.~101018153. For open-access purposes, the authors have applied a CC BY public copyright license to any author-accepted manuscript version arising from this submission.} 

\begin{abstract}
    The existence of global weak solutions to a broad class of Navier--Stokes--Korteweg equations is established for large data in the three-dimensional torus, including the diffuse-interface and quantum Navier--Stokes systems as special cases. The model consists of the compressible Navier--Stokes equations with degenerate density-dependent viscosity and a general nonlinear third-order Korteweg term. The existence proof combines a priori estimates provided by the energy and Bresch--Desjardins (BD) entropy inequalities with a carefully designed approximation scheme. A crucial ingredient of the analysis is a new dissipation inequality associated with the Korteweg term, obtained via the systematic integration-by-parts method. To construct the solutions, artificial drag terms and a quantum Korteweg regularization are introduced, which are subsequently removed in the limit by establishing a renormalized formulation.
\end{abstract}

 

\maketitle

\section{Introduction}

Compressible Navier--Stokes--Korteweg systems provide a general framework for the description of viscous compressible fluids with capillarity effects and arise in a variety of physical contexts, including diffuse-interface \cite{AFW98} and quantum fluid models \cite{Jue12}. A central difficulty in the mathematical analysis of these models stems from the simultaneous presence of nonlinear convection, density-dependent viscosity, and higher-order capillarity terms, often in regimes where vacuum may occur and the viscosity degenerates. The purpose of the present work is to develop a unified existence theory for a broad family of compressible Navier--Stokes--Korteweg equations, encompassing diffuse-interface and quantum fluid models as special cases.

The dynamics of the density $\rho$ and velocity $v$ of the fluid are governed by the equations
\begin{align}
  & \pa_t\rho + \diver(\rho v) = 0\quad\mbox{in }\Omega,\ t>0,
  \label{eq: 1.mass} \\
  & \pa_t(\rho v) + \diver(\rho v\otimes v) + \na p(\rho)
  = \nu\diver(\rho\mathbb{D}(v)) + \diver\mathbb{K}(\rho), \label{eq: 1.momentum}
\end{align}
where $\Omega=\T^3$ is the three-dimensional torus, $\mathbb{D}(v)=\frac12(\na v+\na v^T)$ is the symmetric part of the velocity gradient, and $p(\rho)=\rho^\gamma$ is the pressure with the adiabatic exponent $\gamma>1$. The initial conditions read as
\begin{align}
  \sqrt{\rho}(0,\cdot) = \sqrt{\rho^0}, \quad (\sqrt{\rho} v)(0,\cdot)=\sqrt{\rho^0} v^0
  \quad\mbox{in }\Omega. \label{eq: 1.IC}
\end{align}
The Korteweg stress tensor equals
\begin{align*}
  \mathbb{K}(\rho) = \bigg(\diver(\rho k(\rho)\na\rho)
  - \frac12\big(\rho k'(\rho)+k(\rho)\big)|\na\rho|^2\bigg)\mathbb{I}
  - k(\rho)\na\rho\otimes\na\rho,
\end{align*}
where $k(\rho)$ is the capillarity function \cite[formula (1.29)]{DuSe85}. Our main assumption is that $k(\rho)$ is given by
\begin{align*}
  k(\rho) = \kappa\rho^\alpha,
\end{align*}
where $\kappa>0$ is the capillarity constant and $\alpha\in[-1,0]$. Generally, if $k(\rho)=\psi'(\rho)^2$ for some function $\psi(\rho)$, the third-order Korteweg term simplifies to
\begin{align*}
  \diver\mathbb{K}(\rho) = \rho\na\big(\psi'(\rho)\Delta\psi(\rho)\big).
\end{align*}
The case of constant capillarity $\alpha=0$ leads to $\diver\mathbb{K}(\rho) = 2\kappa\rho\na\Delta\rho$, used in van-der-Waals diffuse-interface models \cite{AFW98}, while the case $\alpha=-1$ gives $\diver\mathbb{K}(\rho) = \kappa\rho\na(\Delta\sqrt\rho/\sqrt\rho)$. The expression $\Delta\sqrt\rho/\sqrt\rho$ is called the quantum Bohm potential exerting a nonlocal quantum force on the particles \cite{Boh52}. The existence of global weak solutions to \eqref{eq: 1.mass}--\eqref{eq: 1.IC} for $\alpha=0$ was proved in \cite{AnSp22}, while global weak solutions for $\alpha=-1$ were shown to exist in \cite{VaYu16}, both for large data and including vacuum. The aim of this paper is to prove a global existence result for all $\alpha\in[-1,0]$, thus generalizing existing results to a broad class of Navier--Stokes--Korteweg models.

\subsection{State of the art}
The analysis of \eqref{eq: 1.mass}--\eqref{eq: 1.momentum} relies on uniform estimates derived from the free energy and Bresch--Desjardins (BD) entropy inequalities, where the free energy $E$ and BD entropy $E_{BD}$ are defined, respectively, by
\begin{align*}
  E(\rho,v) &= \int_{\Omega}\bigg(\frac{\rho}{2}|v|^2 + h(\rho)
  + \frac{2\kappa}{(\alpha+2)^2}|\na\rho^{\alpha/2+1}|^2\bigg)dx, \\
  E_{BD}(\rho,v) &= \frac12\int_{\Omega}\rho|v+\nu\na\log\rho|^2 dx,
\end{align*}
where $h(\rho) = \rho^\gamma/(\gamma-1)$. The free energy is the sum of the kinetic, internal, and capillary energies. The BD entropy was first introduced in \cite{BDL03}. It can be interpreted as the kinetic energy associated to the effective velocity $v+\nu\na\log\rho$, involving the osmotic velocity $\nu\na\log\rho$. The BD entropy can be derived more generally for viscous terms of the type 
\begin{align*}
  \diver\big(\mu_1(\rho)\mathbb{D}(v) 
  + \mu_2(\rho)\diver v\,\mathbb{I}\big),
\end{align*}
where $\mu_2(\rho) = \rho\mu_1'(\rho) - \mu_1(\rho)$ has to be satisfied. Then the effective velocity reads as $v+\na\phi(\rho)$, where $\phi'(\rho) = \mu_1'(\rho)/\rho$. It follows for the choice $\mu_1(\rho)=\nu\rho$ as in equation \eqref{eq: 1.momentum} that $\phi(\rho)=\nu\log\rho$. 

When the viscosity and capillarity coefficients satisfy a certain relation, system \eqref{eq: 1.mass}--\eqref{eq: 1.momentum} can be reformulated in terms of the variable $(\rho,v+\na\xi(\rho))$ for some $\xi$ as a Navier--Stokes system; in particular, the Korteweg term vanishes \cite{Jue11}. The a priori estimates provided by the energy and BD entropy inequalities for \eqref{eq: 1.mass}--\eqref{eq: 1.momentum} do not yield sufficient regularity of the velocity field to adequately address the occurrence of vacuum regions. When $\kappa=0$, it is possible to derive additional bounds from the Mellet--Vasseur estimate \cite{MeVa07}, yielding suitable compactness properties for the sequence of approximating solutions. If $\kappa>0$, it is generally not possible to infer a Mellet--Vasseur-type estimate. Assuming that the viscosity and capillarity coefficients are related in such a way that the Korteweg term vanishes upon reformulation, as discussed above, one can derive a Mellet--Vasseur-type estimate for the new variable. In our case, $\mu_1(\rho)=\rho$, this strategy is only possible for the quantum Navier--Stokes model with $\alpha=-1$, as considered in \cite{AnSp17}. 

The issue of vacuum regions was investigated in several works in the cases $\alpha=0$ and $\alpha=-1$. In case $\alpha=-1$ (quantum Navier--Stokes system), the first global existence result was given in \cite{Jue10} using test functions of the type $\rho\varphi$ and assuming that $\kappa<\nu$. The case $\kappa>\nu$ was considered in \cite{Jia11}. The vacuum difficulty was resolved in \cite{GiLa15} by adding the cold pressure term $\na\rho^{-\beta}$ for some $\beta>0$, while \cite{VaYu16} added the drag term $r_0v+r_1\rho v|v|^2$. These augmented systems ensure that the velocity field is well defined even in the vacuum region. The existence of global weak solutions without cold pressure or drag terms was achieved in \cite{LaVa18} for the quantum Navier--Stokes system. The method is based on the construction of weak solutions that are renormalized in the velocity variable, first used in \cite{VaYu16inv} for the case $\kappa=0$. For $\alpha=0$ (diffuse-interface model), the authors of \cite{BDL03} used test functions of the form $\rho\varphi$ in the momentum equation. In \cite{AnSp22}, the authors establish the existence of solutions for the unmodified system by extending the renormalization method through the introduction of an additional density truncation alongside the velocity renormalization.

For general functions $k(\rho)=\rho^\alpha$ and $\mu(\rho)=\rho^\beta$, existence results for the one-dimensional equations were published in \cite{GeFl16} if  $2/3<\beta\le 1$ and $2\beta-3<\alpha\le -1$. This range was later extended in \cite{ABS25} to $2\beta-3\le\alpha<2\beta-1$, $\alpha>-2$, and $\beta>1/2$, in particular giving $-1\le\alpha<1$ when $\beta=1$. The case $\beta=(\alpha+3)/2$ and $\alpha>-1$ was investigated in \cite{BuHa22}, still in one dimension. If $\beta\ge\gamma>1$ and $\gamma\le\alpha+4$, there exist initial data and a certain range for $\alpha$ such that solutions to \eqref{eq: 1.mass}--\eqref{eq: 1.momentum} in $\R$ blow up in finite time \cite{TWL22}.
 
A major difficulty in the analysis of the Navier--Stokes--Korteweg equations is obtaining sufficient regularity from the BD entropy estimate to control the Korteweg term. In particular, using the test function $\log\rho$ in the momentum equation, we need to derive gradient bounds from the integral
\begin{align}\label{1.J}
  J(\rho) = -\int_\Omega\na\rho\cdot\na\bigg(\rho^{\alpha/2}
  \frac{\Delta\rho^{\alpha/2+1}}{\alpha/2+1}\bigg)dx.
\end{align}
In one space dimension, such an inequality was derived in \cite{GeFl16}, using Sobolev and Gagliardo--Nirenberg inequalities. By exploiting the systematic integration-by-parts method introduced in \cite{JuMa06}, we establish suitable gradient estimates in arbitrary space dimensions, generalizing the one-dimensional results of \cite{GeFl16} to the multi-dimensional case.


\subsection{Key ideas}

As already mentioned, a priori estimates are derived from approximate versions of the energy and BD entropy identities
\begin{align*}
  \frac{d}{dt}E(\rho,v) + \int_\Omega \nu\rho|\mathbb{D}(v)|^2 dx &= 0, \\
  \frac{d}{dt}E_{BD}(\rho,v) + \int_\Omega\bigg(
  \frac{\rho}{2}|\na v-\na v^T|^2 + \frac{4\nu}{\gamma}
  |\na\rho^{\gamma/2}|^2\bigg)dx
  &\leq -\kappa J(\rho),
\end{align*}
where $J(\rho)$ is defined in \eqref{1.J}. The key step is the proof of the inequality
\begin{align}\label{1.ineq}
  J(\rho) \ge C_0\int_\Omega\big(|\na\rho^{\alpha/4+1/2}|^4 
  + |\Delta\rho^{\alpha/2+1}|^2\big)dx
\end{align}
for all $-30/19<\alpha<0$, where $C_0>0$ only depends on $\alpha$ and the space dimension, and which is proved by the systematic integration-by-parts method of \cite{JuMa06}. If $\alpha=0$, we obtain trivially $J(\rho)=\int_\Omega(\Delta\rho)^2 dx$. In one space dimension, the inequality holds for all $-2<\alpha<1$, $\alpha\neq0$, by \cite[Theorem~2.2]{GeFl16}, and this range is optimal. We repeat the proof in Remark \ref{rem.1D} using the systematic integration-by-parts method. In several dimensions, we prove inequality \eqref{1.ineq} for all $-2d(d+2)/(2d^2+1)<\alpha<0$; see Theorem \ref{theorem.ineq}.

A crucial step in the existence proof is the construction of a suitable approximation scheme. We follow the strategy introduced in \cite{AnSp22}. As a first approximation, we add drag terms together with a quantum Korteweg regularization, i.e., the Korteweg tensor corresponding to the exponent $\alpha=-1$,
\begin{align*}
    \diver\K_{Q}(\rho) = 2\rho\nabla \left(\frac{\Delta \sqrt{\rho}}{\sqrt{\rho}}\right)
\end{align*}
and consider, for $\sigma>0$, the regularized system
\begin{align}
    &\partial_t \rho + \diver(\rho v) =0,\label{eq: 1a.mass}\\
    &\partial_t (\rho v) + \diver(\rho v\otimes v) + \nabla p(\rho) = \nu \diver(\rho \D v) + \diver\K(\rho)
    \label{eq: 1a.momentum}\\
    &\phantom{\partial_t (\rho v) + \diver(\rho v\otimes v) + \nabla p(\rho) =}-\sigma v -\sigma \rho |v|^2 v +\sigma \diver\K_{Q}(\rho). \nonumber 
\end{align}
The drag terms ensure that the velocity is well defined in vacuum regions. The quantum Korteweg term is the capillarity term naturally associated with the density-dependent viscosity. This structural compatibility yields a viscosity regularity estimate, which is no longer available once the quantum Korteweg regularization is removed.

In order to solve this approximate system, it is well known that it requires a carefully designed regularization strategy. As in \cite[Sec.~2]{VaYu16}, we add higher-order regularizations with parameters $(\delta,\eta,\eps,\mu)$ such that the modified energy ensures that $\rho$ is bounded away from zero, which is essential for deriving the BD entropy identity. The resulting approximate problem is solved by a Faedo--Galerkin scheme of dimension $n\in\mathbb{N}$, and the limit $n\to\infty$ is performed first. The uniform estimates obtained from the energy and BD entropy identities then provide sufficient compactness to pass to the limit $(\delta,\eta,\varepsilon,\mu) \to 0$ as in \cite{VaYu16}. More precisely, we show that
\begin{align}\label{1.BD2}
  \frac{d}{dt}E_{BD}(\rho,v) + P(\rho,v) 
  \le C_1(\sigma) + \bigg(\eps + \mu + \frac{\eps}{\sqrt\mu}\bigg)
  C_2(\sigma,\delta,\eta),
\end{align}
where $P(\rho,v)$ is the entropy production term containing gradient estimates for $\rho$, $C_1(\sigma)>0$ is independent of $(\delta,\eta,\varepsilon,\mu)$, and $C_2(\sigma,\delta,\eta)>0$ is independent of $(\eps,\mu)$. First, we perform the limit $(\eps,\mu)\to 0$ such that $\eps/\sqrt\mu\to 0$. As a consequence, the right-hand side of \eqref{1.BD2} is bounded uniformly in $(\delta,\eta)$, and in the second step, we can pass to the limit $(\delta,\eta)\to 0$.

The final limit $\sigma\to0$ cannot be justified directly at the level of weak solutions. Instead, it is necessary to work with renormalized weak solutions, similarly as in \cite{AnSp22}: The solutions are renormalized with respect to the velocity and truncated with respect to the density. Consequently, we have to prove that the constructed weak solution satisfies the corresponding renormalized formulation. Compared with the approach of \cite{LaVa18}, the density truncation is additionally required to establish the equivalence between weak and renormalized weak solutions.


\subsection{Main result}

We introduce the tensor
\begin{align}
  \sqrt{\nu\rho}\T_\nu &= \nu\na(\rho v) 
  - 2\nu\sqrt{\rho} v\otimes\na\sqrt{\rho}. \label{eq: 1.Tnu}
\end{align}
For positive smooth densities $\rho$, the tensor $\T_\nu$ equals $\sqrt{\nu\rho}\na v$. We denote the symmetric part of $\T_\nu$ by $\S_\nu = \frac12(\T_\nu+\T_\nu^T)$, and we set $\dom_T:=(0,T)\times\dom$. 

\begin{definition}[Weak solution to the Navier--Stokes--Korteweg system]
The tuple $(\sqrt{\rho}$, $\sqrt{\rho}v)$ is called a weak solution to \eqref{eq: 1.mass}--\eqref{eq: 1.momentum} on $[0,T]$ with initial conditions \eqref{eq: 1.IC} if the regularity 
\begin{equation}\label{eq: 1.regularity}
\begin{aligned}
  & \sqrt{\rho} v \in L^{\infty}(0,T; L^2(\dom;\R^3)), \quad 
  &&\rho \in L^{\infty}(0,T; L^\gamma(\dom)), \quad
  &&\nabla \sqrt{\rho} \in L^{\infty}(0,T; L^2(\dom)), \\
  & \nabla \rho^{\alpha/2+1}\in L^{\infty}(0,T; L^2(\dom)), \quad
  && \nabla\rho^{\gamma/2}\in L^2(\dom_T), \quad 
  &&\T_{\nu} \in L^{2}(\dom_T), \\
  &\mbox{and for }\alpha\in[-1,0): 
  && \nabla \rho^{\alpha/4+1/2}\in L^4(\dom_T), \quad 
  &&\Delta \rho^{\alpha/2+1} \in L^2(\dom_T)
\end{aligned}
\end{equation}
holds and if for all $\theta \in C_0^\infty([0,T)\times\dom)$, $\psi \in C_0^\infty([0,T)\times\dom;\R^3)$,
\begin{align}
    &\int_{\dom} \rho^0 \theta(0,\cdot) dx + \int_0^T\int_{\dom} \rho \partial_t \theta dxdt +\int_0^T\int_{\dom} \rho v \cdot \nabla \theta dxdt = 0, \label{eq: 1.wmass}\\
    &\int_{\dom}\rho^0 v^0\cdot \psi(0,\cdot)dx + \int_0^T \int_{\dom}\big(\rho v \cdot \partial_t \psi + (\rho v \otimes v): \nabla \psi - \nabla p(\rho) \cdot \psi\big) dx dt \label{eq: 1.wmomentum}\\
    &= \int_0^T \int_{\dom}\bigg( \sqrt{\nu \rho} \S_{\nu}: \nabla \psi + \kappa \frac{\Delta \rho^{\alpha/2+1}}{(\alpha/2+1)^2}\nabla\rho^{\alpha/2+1} \cdot \psi + \kappa \frac{\Delta\rho^{\alpha/2+1}}{\alpha/2+1}\rho^{\alpha/2+1} \diver\psi\bigg)dxdt, \nonumber
\end{align}
where $\S_{\nu}= \frac{1}{2}(\T_{\nu} +\T_{\nu}^T)$, and $\T_{\nu}$ satisfies \eqref{eq: 1.Tnu} a.e.\ on $\dom_T$.
\end{definition}

We impose the following conditions on the initial data:
\begin{align}\label{eq: 1.regIC}
    \sqrt{\rho^0} \in L^{2\gamma}(\Omega), \quad \sqrt{\rho^0},\,(\rho^0)^{\alpha/2+1} \in H^1(\Omega), \quad \sqrt{\rho^0} v^0 \in L^2(\dom).
\end{align}

\begin{theorem}\label{theorem.ex}
    Let $\alpha\in[-1,0]$, $\gamma>1$, $\nu>0$, $\kappa>0$, $T>0$, and let $(\sqrt{\rho^0}, \sqrt{\rho^0}v^0)$ satisfy \eqref{eq: 1.regIC}. Then there exists a weak solution to \eqref{eq: 1.mass}--\eqref{eq: 1.momentum} on $[0,T]$.
\end{theorem}

The range of $\alpha$ cannot be easily extended to $\alpha<-1$; see Remark \ref{rem.alpha}. Furthermore, the generalization to viscosity coefficients $\mu_1(\rho)=\rho^\beta$ with $\beta>0$ is not straightforward; see Remark \ref{rem.visc}.


\subsection{Organization of the paper}

The paper is organized as follows. In Section~\ref{sec.prep}, we prove inequality \eqref{1.ineq} and give the definition of weak and renormalized solutions to the regularized Navier--Stokes--Korteweg system \eqref{eq: 1a.mass}--\eqref{eq: 1a.momentum}. Section~\ref{sec:wksol} is devoted to the existence of weak solutions to the regularized system with drag and quantum Korteweg terms. In Section~\ref{sec.wkTOr}, we show that weak solutions to the regularized system \eqref{eq: 1a.mass}--\eqref{eq: 1a.momentum} are renormalized weak solutions. Conversely, in Section~\ref{sec.rTOwk}, we prove that renormalized weak solutions are weak solutions independently of the presence of the regularization terms. The limit in the regularization parameter $\sigma\to 0$ is performed in Section~\ref{sec.sigma}. Finally, we discuss possible generalizations and limitations in Section \ref{sec.rem}.


\section{Preparations}\label{sec.prep}

Before proving our main result, Theorem \ref{theorem.ex}, we first introduce the notions of weak and renormalized weak solutions to the regularized Navier--Stokes--Korteweg system and establish inequality \eqref{1.ineq}.

\subsection{Definitions}

The proof of Theorem \ref{theorem.ex} relies on an approximate Navier--Stokes--Korte\-weg system and associated weak and renormalized solutions, which are defined in this section. We recall the definition of the quantum Korteweg tensor, written as
\begin{align}
    \K_{Q}(\rho) = \sqrt{\rho}\big(\na^2\sqrt{\rho} - 4\na\sqrt[4]{\rho}
  \otimes\na\sqrt[4]{\rho}\big). \label{eq: 1.SQT} 
\end{align}

\begin{definition}[Weak solution to the regularized Navier--Stokes--Korteweg system]\label{def.weakregul}
The tuple $(\sqrt{\rho}, \sqrt{\rho}v)$ is called a weak solution to \eqref{eq: 1a.mass}--\eqref{eq: 1a.momentum} on $[0,T]$ with initial conditions \eqref{eq: 1.IC} if the regularity \eqref{eq: 1.regularity} as well as
\begin{equation}\label{1.regul}
\begin{aligned}
  \sqrt[4]{\sigma} \nabla \sqrt[4]{\rho}, \, \sqrt[4]{\sigma}\sqrt[4]{\rho}v \in L^4(\dom_T), \ \sqrt{\sigma} \Delta \sqrt{\rho}, \, \sqrt{\sigma} v \in L^2(\dom_T),\
  \sigma\int_{\dom} (\log\rho(t,x))_{-} dx\leq C
  \end{aligned}
\end{equation}
for all $0<t<T$ hold, where $z_-=\max\{0,z\}$, and if for all $\theta \in C_0^\infty([0,T)\times\dom)$, $\psi \in C_0^\infty([0,T)\times\dom;\R^3)$,
\begin{align*}
    &\int_{\dom} \rho^0 \theta(0,\cdot) dx + \int_0^T\int_{\dom} \rho \partial_t \theta dxdt +\int_0^T\int_{\dom} \rho v \cdot \nabla \theta dxdt = 0, \\ 
    &\int_{\dom}\rho^0 v^0\cdot \psi(0,\cdot)dx + \int_0^T \int_{\dom}\big(\rho v \cdot \partial_t \psi + (\rho v \otimes v): \nabla \psi - \nabla p(\rho) \cdot \psi\big) dx dt 
    \\ 
    &= \int_0^T \int_{\dom}\bigg( \sqrt{\nu \rho} \S_{\nu}: \nabla \psi + \kappa \frac{\Delta \rho^{\alpha/2+1}}{(\alpha/2+1)^2}\nabla\rho^{\alpha/2+1} \cdot \psi + \kappa \frac{\Delta\rho^{\alpha/2+1}}{\alpha/2+1}\rho^{\alpha/2+1} \diver\psi\bigg)dxdt\nonumber\\
    &\phantom{xx}
    +\sigma \int_0^T \int_{\dom}\big( v\cdot \psi + \rho v |v|^2 \cdot \psi + \K_{Q}: \nabla \psi\big) dxdt, \nonumber
\end{align*}
where $\K_{Q}$ is defined in \eqref{eq: 1.SQT}, $\S_{\nu}= \frac{1}{2}(\T_{\nu} +\T_{\nu}^T)$, and $\T_{\nu}$ satisfies \eqref{eq: 1.Tnu} a.e.\ on $\dom_T$.
\end{definition}

\begin{definition}[Renormalized solution to the regularized  Navier--Stokes--Korteweg system]\label{def.renormregul}
    The tuple $(\sqrt{\rho}, \sqrt{\rho}v)$ is called a renormalized weak solution to \eqref{eq: 1a.mass}--\eqref{eq: 1a.momentum} on $[0,T]$ with initial conditions \eqref{eq: 1.IC} if the regularity \eqref{eq: 1.regularity} and \eqref{1.regul} is satisfied and the following properties hold:
    For every $\varphi\in W^{3,\infty} (\R^3)$ such that there exists $C_1>0$ satisfying
\begin{align}\label{eq: 1.varphi}
  |y_j\varphi(y)| + \bigg|y_j\frac{\pa\varphi}{\pa y_k}(y)\bigg|
  \le C_1\quad\mbox{for all }y\in\R^3,\ j,k=1,2,3,
\end{align}
and for every $\zeta\in W^{2,\infty}(\R)$ with compact support and $    \supp(\zeta')\subset [-C_2, -1/C_2]\cup [1/C_2, C_2]$ for some $C_2>0$,
there exist measures $R_1, R_2, Q_{ijk}\in \mathcal{M}(\Omega_T)$ such that
\begin{align}
    \|R_1&\|_{\mathcal{M}(\Omega_T)} + \sigma\|R_2\|_{\mathcal{M}(\Omega_T)} + \sum_{i,j,k=1}^3\|Q_{ijk}\|_{\mathcal{M} (\Omega_T)} \nonumber \\
    &\leq C\big(\|(\cdot)^{-\alpha/2}\zeta'(\cdot)\|_{L^{\infty}} \|\varphi\|_{L^{\infty}}  + \|\varphi\|_{L^{\infty}}\|\zeta'\|_{L^{\infty}} + \|\varphi'\|_{L^\infty} \|\zeta'\|_{L^\infty} + \|\zeta\|_{L^{\infty}}\|\varphi''\|_{L^{\infty}}
    \label{eq: 1.boundMeasure} \\
    &\phantom{xx}+ \|\varphi''\|_{L^{\infty}}\|(\cdot)^{\alpha/2+1/2} \zeta(\cdot)\|_{L^{\infty}}\big) + C \sqrt{\sigma} \big( \|\sqrt{\cdot} \zeta'(\cdot)\|_{L^{\infty}} 
    + \|\varphi'\|_{L^{\infty}}+\|\zeta\|_{L^{\infty}} \|\varphi''\|_{L^{\infty}}\big) \nonumber \\
  &\phantom{xx}+ C\|\varphi''\|_{L^{\infty}}\nonumber
\end{align}
for some constant $C>0$, only depending on the norms associated to the regularity \eqref{eq: 1.regularity} and \eqref{1.regul}, such that the following equations are satisfied for all $\theta\in C^\infty_0([0,T)\times \dom)$:
\begin{align}
    &\int_{\dom} \rho^0 \theta(0,\cdot) dx + \int_0^T\int_{\dom} \rho \partial_t \theta dxdt +\int_0^T\int_{\dom} \rho v \cdot \nabla \theta dxdt = 0, \label{eq: 1.rmass}\\
    &\int_{\dom}\rho^0 v^0\cdot \varphi'(v^0)\zeta(\rho^0)\theta(0,\cdot)dx  + \int_0^T \int_{\dom}\rho\varphi(v)(\partial_t \theta + v \cdot \nabla \theta)\zeta(\rho)  dx dt \label{eq: 1.rmomentum}\\
    &-\int_0^T \int_{\dom} \sqrt{\nu \rho} \S_{\nu} : (\varphi'(v)\otimes \nabla \theta)\zeta(\rho) dxdt 
    - \int_0^T\int_{\dom} \nabla p(\rho)\cdot \varphi'(v) \theta \zeta(\rho) dxdt\nonumber\\ 
    &-\int_0^T \int_{\dom} \kappa
    \frac{\Delta \rho^{\alpha/2+1}}{\alpha/2+1}\varphi'(v)
    \cdot\bigg(\rho^{\alpha/2+1}\nabla \theta 
    + \frac{\nabla \rho^{\alpha/2+1}}{\alpha/2+1}\theta
    \bigg)\zeta(\rho)dxdt \nonumber\\
    &-\sigma\int_0^T \int_{\dom}(v + \rho |v|^2v)
    \cdot\varphi'(v)\zeta(\rho)\theta  dxdt
    -\sigma\int_0^T\int_{\dom} \K_{Q} :(\varphi'(v)\otimes \nabla \theta)\zeta(\rho) dxdt\nonumber\\
    &+\int_0^T \int_{\dom} (R_1 + \sigma R_2)\theta 
    dxdt = 0,\nonumber
\end{align}
$\K_{Q}$ is defined in \eqref{eq: 1.SQT}, $\S_{\nu}= \frac{1}{2}(\T_\nu+\T_\nu^T)$, and $\T_\nu$ satisfies for all $\theta\in C_0^\infty([0,T)\times\dom)$:
\begin{align}
    \int_0^T\int_\Omega\sqrt{\nu\rho}\frac{\pa\varphi}{\pa v_i}(v)
  (\mathbb{T}_\nu)_{jk}\theta dxdt
  &= -\nu\int_0^T\int_\Omega\bigg(\rho\frac{\pa\theta}{\pa x_j}
  + 2\sqrt{\rho}\frac{\pa\sqrt\rho}{\pa x_j}\theta\bigg)
  \frac{\pa\varphi}{\pa v_i}(v)v_k dxdt \label{eq: 1.rTnu}\\
  &\phantom{xx}+ \int_0^T\int_{\dom} Q_{ijk} \theta dxdt
  \quad\mbox{for }i,j,k=1,2,3. \nonumber 
\end{align}
\end{definition}


\subsection{An auxiliary inequality}\label{sec.ineq}

The Korteweg term is estimated using the following inequality.

\begin{theorem}\label{theorem.ineq}
Let $d\in\N$, $d>1$, and 
\begin{align*}
  -\frac{2d(d+2)}{2d^2+1} <\alpha < 0.
\end{align*}
Then there exists $C_0=C_0(\alpha,d)>0$ such that for all positive smooth functions $\rho$,
\begin{align}\label{2.in}
  -\int_{\T^d} \na\rho\cdot\na
  \bigg(\rho^{\alpha/2}\frac{\Delta\rho^{\alpha/2+1}}{
  \alpha/2+1}\bigg)dx \ge C_0\int_{\T^d}
  \big(|\na\rho^{\alpha/4+1/2}|^4 + |\Delta\rho^{\alpha/2+1}|^2\big)dx.
\end{align} 
\end{theorem} 

In three space dimensions, the admissible range is $-30/19<\alpha<0$. If $\alpha=0$, we obtain
\begin{align*}
  -\int_{\T^d} \na\rho\cdot\na
  \bigg(\rho^{\alpha/2}\frac{\Delta\rho^{\alpha/2+1}}{
  \alpha/2+1}\bigg)dx 
  = \int_{\T^d}(\Delta\rho)^2 dx 
  = \int_{\T^d}(\Delta\rho^{\alpha/2+1})^2 dx.
\end{align*}
Therefore, we can reformulate inequality \eqref{2.in} for all  $-2d(d+2)/(2d^2+1) <\alpha \le 0$ as
\begin{align}\label{2.ineq} 
   -\int_{\T^d} \na\rho\cdot\na
  \bigg(\rho^{\alpha/2}\frac{\Delta\rho^{\alpha/2+1}}{
  \alpha/2+1}\bigg)dx \ge C_0\int_{\T^d}
  \big(\mathrm{1}_{\alpha<0}|\na\rho^{\alpha/4+1/2}|^4 
  + |\Delta\rho^{\alpha/2+1}|^2\big)dx,
\end{align}
where $\mathrm{1}_{\alpha<0}$ equals one if $\alpha<0$ and zero if $\alpha=0$. We discuss the one-dimensional case $d=1$ in Remark \ref{rem.1D}.

\begin{proof}[Proof of Theorem \ref{theorem.ineq}]
We apply the systematic integration-by-parts technique introduced in  \cite{JuMa06} and applied in \cite{JuMa08}. For a detailed exposition of this technique, we refer to \cite[Chap.~3]{Jue16}. To simplify the notation, we set $\beta:=\alpha/2$. We integrate by parts on the left-hand side of \eqref{2.ineq}:
\begin{align*}
  L &:= -\int_{\T^d} \na\rho\cdot\na
  \bigg(\rho^{\alpha/2}\frac{\Delta\rho^{\alpha/2+1}}{\alpha/2+1}
  \bigg)dx
  = \int_{\T^d}\rho^\beta\Delta \rho
  \frac{\Delta\rho^{\beta+1}}{\beta+1}dx \\
  &= \int_{\T^d}\bigg(\frac{\Delta\rho^{\beta+1}}{\beta+1}
  - \beta\rho^{\beta-1}|\na\rho|^2\bigg)
  \frac{\Delta\rho^{\beta+1}}{\beta+1}dx \\
  &= \int_{\T^d}\bigg(\frac{(\Delta\rho^{\beta+1})^2}{(\beta+1)^2}
  - \frac{\beta}{(\beta+1)^3}
  \frac{|\na\rho^{\beta+1}|^2}{\rho^{\beta+1}}
  \Delta\rho^{\beta+1}\bigg)dx.
\end{align*}
Setting $\lambda=\rho^{\beta+1}$, we write the left- and right-hand sides of \eqref{2.in}:
\begin{align*}
  L = \int_{\T^d}\bigg(\frac{(\Delta\lambda)^2}{(\beta+1)^2}
  - \frac{\beta}{(\beta+1)^3}
  \frac{|\na\lambda|^2}{\lambda}\Delta\lambda\bigg), \quad
  R = \int_{\T^d}\bigg(\frac{|\na\lambda|^4}{16\lambda^2}
  + (\Delta\lambda)^2\bigg)dx.
\end{align*}
We introduce the scalar variables
\begin{align*}
  \xi_G = \frac{|\na\lambda|}{\lambda}, \quad
  \xi_L = \frac{\Delta\lambda}{\lambda}, \quad
  \xi_H = \frac{|\na^2\lambda|}{\lambda}, \quad
  \xi_{GHG} = \frac{1}{\lambda^3}\na\lambda^T\na^2\lambda\na\lambda,
\end{align*}
which yields
\begin{align*}
  L = \int_{\T^d}\lambda^2\bigg(\frac{\xi_L^2}{(\beta+1)^2} 
  - \frac{\beta}{(\beta+1)^3}\xi_G^2\xi_L\bigg)dx, \quad
  R = \int_{\T^d}\lambda^2\bigg(\frac{1}{16}\xi_G^4 + \xi_L^2\bigg)dx.
\end{align*}
Our computation relies, as in \cite[Sec.~2]{JuMa08}, on two ``dummy'' integrals, which are formal expressions that vanish:
\begin{align*}
  J_1 &= \int_{\T^d}\diver\big(\lambda^{-1}|\na\lambda|^2\na\lambda
  \big)dx = \int_{\T^d}\lambda^2(-\xi_G^4 + 2\xi_{GHG} + \xi_G^2\xi_L)
  dx = 0, \\
  J_2 &= \int_{\T^d}\diver\big((\na^2\lambda 
  - \Delta\lambda\mathbb{I})\na\lambda\big)dx
  = \int_{\T^d}\lambda^2(\xi_H^2-\xi_L^2)dx = 0.
\end{align*}
Inequality \eqref{2.in} is shown if we find constants $C_0>0$ and $C_1$, $C_2\in\R$ such that the integrand of $L - C_0R + C_1J_1 + C_2J_2$ is nonnegative. To this end, setting $K_i = C_i(\beta+1)^3$ for $i=0,1,2$, we write
\begin{align*}
  L - C_0R + C_1J_1 + C_2J_2 &= \int_{\T^d}\frac{\lambda^2}{(\beta+1)^3}
  \bigg\{-\bigg(\frac{K_0}{16}+K_1\bigg)\xi_G^4
  + (\beta+1-K_0-K_2)\xi_L^2 \\
  &\phantom{xx}+ (K_1-\beta)\xi_G^2\xi_L
  + 2K_1\xi_{GHG} + K_2\xi_H^2\bigg\}dx.
\end{align*}
It is convenient to replace the variables $(\xi_H,\xi_{GHG})$ by $(\xi_R,\xi_S)$, defined by
\begin{align*}
  (d-1)\xi_G^2\xi_S = \xi_{GHG} - \frac{1}{d}\xi_G^2\xi_L, \quad
  \xi_H^2 = \frac{1}{d}\xi_L^2 + d(d-1)\xi_S^2 + \xi_R^2.
\end{align*}
The variable $\xi_R^2$ is well defined, since \cite[Lemma 2.1]{JuMa08}
\begin{align*}
  \xi_H^2 \ge \frac{1}{d}\xi_L^2 + \frac{d}{d-1}
  \bigg(\frac{\xi_{GHG}}{\xi_G^2} - \frac{1}{d}\xi_L\bigg)^2
  = \frac{1}{d}\xi_L^2 + d(d-1)\xi_S^2,
\end{align*}
and thus, $\xi_R^2$ measures the positive defect of this inequality. This gives
\begin{align*}
  L -{}& C_0R + C_1J_1 + C_2J_2 \\
  &= \int_{\T^d}\frac{\lambda^2}{(\beta+1)^3}
  \bigg\{-\bigg(\frac{K_0}{16}+K_1\bigg)\xi_G^4
  + \bigg(\beta+1-K_0-K_2\bigg(1-\frac{1}{d}\bigg)\bigg)\xi_L^2 \\
  &\phantom{xx}+ \bigg(\bigg(1+\frac{2}{d}\bigg)K_1-\beta\bigg)
  \xi_G^2\xi_L
  + 2(d-1)K_1\xi_G^2\xi_S + d(d-1)K_2\xi_S^2 + K_2\xi_R^2\bigg\}dx.
\end{align*}
We choose $K_1$ and $K_2$ in such a way that the variables $\xi_L^2$ and $\xi_G^2\xi_L$ are eliminated:
\begin{align*}
  K_1 = \frac{\beta}{1+2/d}, \quad
  K_2 = \frac{d}{d-1}(\beta+1-K_0).
\end{align*}
Then we obtain
\begin{align*}
  L - C_0R + C_1J_1 + C_2J_2 = \int_{\T^d}\frac{\lambda^2}{(\beta+1)^3}
  \big(b_G\xi_G^4 + b_{GS}\xi_G^2\xi_S + b_S\xi_S^2
  + b_R\xi_R^2\big)dx, 
\end{align*}
where
\begin{align*}
  & b_G = -\bigg(\frac{K_0}{16} + \frac{\beta}{1+2/d}\bigg), \quad
  b_{GS} = \frac{2(d-1)\beta}{1+2/d}, \\
  & b_S = d^2(\beta+1-K_0), \quad b_R = \frac{b_S}{d(d-1)}.
\end{align*}
The variable $\xi_R$ only appears once, so its coefficient needs to be nonnegative, which holds if $\beta+1-K_0\ge 0$. The remaining part of the polynomial depending on $(\xi_G,\xi_S)$ can be formulated as
\begin{align*}
  \begin{pmatrix} \xi_S & \xi_G^2\end{pmatrix}
  \begin{pmatrix} b_S & b_{GS}/2 \\ b_{GS}/2 & b_G \end{pmatrix}
  \begin{pmatrix} \xi_S \\ \xi_G^2 \end{pmatrix}.
\end{align*} 
The matrix is positive definite if and only if $b_S>0$ and $b_Gb_S-b_{GS}^2/4 > 0$. The first condition implies that $K_0<\beta+1$, while the second one reads as
\begin{align*}
  F(K_0) := -d^2\bigg(\frac{K_0}{16} + \frac{\beta}{1+2/d}\bigg)
  (\beta+1-K_0) - \beta^2\bigg(\frac{d-1}{1+2/d}\bigg)^2
  > 0.
\end{align*}
If $F(0)>0$, we can find by continuity a positive number $K_0$ such that $F(K_0)>0$. The condition $F(0)>0$ is equivalent to $-d(d+2)/(2d^2+1)<\beta<0$, finishing the proof.
\end{proof}


\section{Construction of weak solutions with $\sigma>0$}\label{sec:wksol}

The aim of this section is the proof of the global existence of a weak solution to the regularized Navier--Stokes--Korteweg system \eqref{eq: 1a.mass}--\eqref{eq: 1a.momentum}. This is done by further approximating the mass and momentum equations.

\subsection{Weak solution of an approximate system}

We prove the existence of a weak solution to the approximate system
\begin{align}
    &\partial_t \rho + \diver(\rho v) = \eps \Delta \rho
    \quad\mbox{in }\dom,\ t>0, \label{eq: 2.mass}\\
    &\partial_t (\rho v) + \diver(\rho v\otimes v) + \nabla p(\rho) = \nu \diver(\rho \D v) + \diver\K(\rho) \nonumber \\
    &\phantom{\partial_t (\rho v)} -\sigma v -\sigma \rho |v|^2 v +\sigma \diver\K_{Q}(\rho)\nonumber\\
    &\phantom{\partial_t (\rho v)} + \eps\diver(v\otimes\na\rho) 
  - \mu\Delta^2 v + \eta\na\rho^{-3} + \delta\rho\na\Delta^5\rho,\nonumber
\end{align}
where $\eps$, $\mu$, $\delta$, $\eta>0$, and we impose smooth initial conditions \eqref{eq: 1.IC} satisfying \eqref{eq: 1.regIC} such that $\rho^0$ is bounded away from zero. The higher-order regularization with parameters $(\delta,\eta,\eps,\mu)$ is analogous to that one introduced in \cite[(2.6)]{VaYu16}. The diffusion term $\eps\Delta\rho$ in the continuity equation ensures better analytical properties for the mass equation, while the correction term $\eps\diver(v\otimes\na\rho)$ is included to preserve the underlying energy structure. The fourth-order viscosity term $\mu\Delta^2 v$ provides additional regularity for the velocity field, which is essential for deriving the BD entropy inequality. Finally, the terms $\eta\na\rho^{-3}$ and $\delta\rho\na\Delta^5\rho$ enforce the strict positivity of the density.

The local existence of a solution is shown by the Faedo--Galerkin method of \cite[Chap.~7]{Fei04} or \cite[Sec.~7.7]{NoSt04}. Applying a fixed-point argument on a finite-dimensional space $X_n\subset L^2(\Omega;\R^3)$ with dimension $n\in\N$, we obtain the existence of a time $T_n>0$ and a solution $(\rho_n,v_n)\in C^1([0,T_n];C^\infty(\Omega))\times C^1([0,T_n];X_n)$ to \eqref{eq: 2.mass} and 
\begin{align}\label{eq: 2.vn}
  \int_\Omega\big(&\rho_n v_n\cdot\psi(T_n,\cdot)
  - \rho^0v^0\cdot\psi(0,\cdot)\big)dx 
  = \int_0^{T_n}\int_\Omega\bigg(\rho_nv_n\cdot\pa_t\psi
  + \rho_n(v_n\otimes v_n):\na\psi \\
  &+ p(\rho_n)\diver\psi 
  - \nu \diver(\rho_n\mathbb{D}(v_n))\cdot\psi
  \bigg)dx dt \nonumber \\
  &+ \int_0^{T_n}\int_\Omega\bigg(\diver\mathbb{K}(\rho_n)
  - \sigma v_n - \sigma \rho_n|v_n|^2v_n + \sigma \diver\K_{Q}(\rho_n)\bigg)\cdot\psi dxdt
  \nonumber \\
  &- \int_0^{T_n}\int_\Omega\big(\eps(v_n\otimes\na\rho_n):\na\psi
  + \mu v_n\cdot\Delta^2\psi + \eta\rho_n^{-3}\diver\psi
  - \delta\rho_n\na\Delta^5\rho_n\cdot\psi\big)dxdt \nonumber 
\end{align}
for all $\psi\in C^1([0,T_n];X_n)$, where $\rho_n$ is bounded from above and below by positive constants depending on the $L^\infty(\Omega)$ norm of $\diver v_n$. The initial conditions are $\rho_n(0,\cdot)=\rho^0$ and $v_n(0,\cdot)=P_nv^0$ in $\Omega$, where $P_n$ is the projection on $X_n$.

We wish to construct a solution on the full interval $[0,T]$. This is achieved by deriving uniform estimates for $v_n$, obtained from the approximate energy
\begin{align*}
  E_{\delta,\eta}(\rho,v) = \int_\Omega\bigg(
  \frac{\rho}{2}|v|^2 + h(\rho) 
  + \frac{2\kappa}{(\alpha+2)^2}|\na\rho^{\alpha/2+1}|^2+ 2\sigma|\nabla\sqrt{\rho}|^2
  + \frac{\eta}{4}\rho^{-3} +\frac{\delta}{2}|\na\Delta^2\rho|^2
  \bigg)dx.
\end{align*}
We use the test function $\psi=v_n$ in the momentum equation \eqref{eq: 2.vn} and the test function $\theta=\frac12|v_n|^2 + h'(\rho_n)  - \kappa\rho_n^{\alpha/2}\Delta\rho_n^{\alpha/2+1}/(\alpha/2+1)-2\sigma \Delta \sqrt{\rho_n}/\sqrt{\rho_n} - \delta\Delta^5\rho_n - \frac34\eta\rho_n^{-4}$ in the mass equation \eqref{eq: 2.mass}. A standard computation yields
\begin{align}\label{eq: 3.ei}
  E_{\delta,\eta}&((\rho_n,v_n)(t))
  + \int_0^t\int_\Omega\bigg(
  \sigma|v_n|^2 + \sigma\rho_n|v_n|^4 + \nu\rho_n|\mathbb{D}(v_n)|^2
  \bigg)dxd\tau \\
  &\phantom{xx}+ \int_0^t\int_\Omega\bigg(
  \frac{4\eps}{\gamma}|\na\rho_n^{\gamma/2}|^2
  + \mu|\Delta v_n|^2 + \frac{4\eps\eta}{3}|\nabla\rho^{-3/2}|^2
  + \eps\delta|\Delta^3\rho_n|^2\bigg)dxd\tau \nonumber \\
  &= E_{\delta,\eta}((\rho_n,v_n)(0))
  + \eps\kappa\int_0^t\int_\Omega \na\rho_n\cdot\na
  \bigg(\rho_n^{\alpha/2}\frac{\Delta\rho_n^{\alpha/2+1}}{
  \alpha/2+1}\bigg)dxd\tau \nonumber \\
  &\phantom{xx}+2\eps \sigma\int_0^t \int_{\dom} \nabla \rho_n \cdot \nabla \bigg(\frac{\Delta \sqrt{\rho_n}}{\sqrt{\rho_n}}\bigg) dxdt \nonumber \nonumber\\
  &\le E_{\delta,\eta}((\rho_n,v_n)(0))
  - C_0\eps\kappa\int_0^t\int_\Omega\big(\mathrm{1}_{\alpha<0}
  |\na\rho_n^{\alpha/4+1/2}|^4
  + |\Delta\rho_n^{\alpha/2+1}|^2\big)dxd\tau \nonumber \\
  &\phantom{xx}-\widetilde{C}_0 \eps \sigma \int_0^t \int_\dom (|\nabla \rho^{1/4}|^4 + |\Delta \sqrt{\rho}|^2)dxdt, \nonumber
\end{align}
where the last step follows from inequality \eqref{2.ineq}. Notice that we lose the gradient bound for $\rho^{\alpha/4+1/2}$ if $\alpha=0$. Inequality \eqref{eq: 3.ei} yields a priori estimates uniform in $n$ and $(\delta,\eta,\eps,\mu)$.

\begin{lemma}
Let $-30/19<\alpha\le 0$. There exists a constant $C>0$, depending on $(\rho^0,v^0)$ but independent of $n$ and $(\delta,\eta,\eps,\mu)$, such that for $(\rho,v)=(\rho_n,v_n)$ and $T = T_n$,
\begin{align}\nonumber 
  &\|\sqrt\rho v\|_{L^\infty(0,T;L^2(\Omega))}
  + \|\rho\|_{L^\infty(0,T;L^\gamma(\Omega))}
  + \|\rho^{\alpha/2+1}\|_{L^\infty(0,T;H^1(\Omega))}
  + \sqrt\sigma\|\sqrt{\rho}\|_{L^\infty(0,T;H^1(\Omega))} \\
  &\phantom{xx}+ \sqrt{\sigma}\|v\|_{L^2(\Omega_T)} + \sqrt[4]{\sigma}\|\sqrt[4]{\rho}v\|_{L^4(\Omega_T)}
  + \sqrt\nu\|\sqrt\rho\mathbb{D}(v)\|_{L^2(\Omega_T)}
  \le C, \nonumber \\
  &\eta^{1/3}\|\rho^{-1}\|_{L^\infty(0,T;L^3(\Omega))}
  + \sqrt\delta\|\rho\|_{L^\infty(0,T;H^5(\Omega))}
  + \sqrt\eps\|\rho^{\gamma/2}\|_{L^2(0,T;H^1(\Omega))} \label{3.est2} \\
  &\phantom{xx}+ \sqrt\mu\|v\|_{L^2(0,T;H^2(\Omega))}
  + \sqrt{\eps\delta}\|\rho\|_{L^2(0,T;H^6(\Omega))}
  + \sqrt{\eps\kappa}\|\rho^{\alpha/2+1}\|_{L^2(0,T;H^2(\Omega))} \nonumber \\
  &\phantom{xx}
  +\sqrt{\eps \sigma}\|\sqrt{\rho}\|_{L^2(0,T;H^2(\Omega))}+\sqrt[4]{\eps\sigma}\|\sqrt[4]{\rho}\|_{L^4(0,T;W^{1,4}(\Omega))}
  \le C. \nonumber 
\end{align}
If $\alpha<0$, we also have the bound
\begin{align*}
  \sqrt[4]{\eps\kappa}\|\rho^{\alpha/4+1/2}
  \|_{L^4(0,T;W^{1,4}(\Omega))} \le C.
\end{align*}
\end{lemma}
These estimates allow us to extend the local solution to a global one on $[0,T]$.

\begin{lemma}\label{lem.n}
Let $-30/19<\alpha\le 0$. For every $T>0$, there exists a constant $C(\eps)>0$, independent of $n$  and $(\delta,\eta,\mu)$ but depending on $\eps$, such that
\begin{align*}
  \|\rho_n^{\alpha/2+1}\|_{L^{10}(\Omega_T)}+ \|\na\rho_n^{\alpha/2+1}\|_{L^{10/3}(\Omega_T)}
  &\le C(\eps).
\end{align*}
\end{lemma}
\begin{proof}
    We know from the energy inequality \eqref{eq: 3.ei} that $(\rho_n^{\alpha/2+1})$ is bounded in $L^2(0,T;H^2(\Omega))$ $\cap L^\infty(0,T;H^1(\Omega))$. Then, by the Gagliardo--Nirenberg inequality with $\theta=(q-6)/(2q)$ and $q>6$, 
\begin{align*}
  \|\rho_n^{\alpha/2+1}\|_{L^q(\Omega_T)}^q
  &\le C\int_0^T\|\rho_n^{\alpha/2+1}\|_{H^2(\Omega)}^{\theta q}
  \|\rho_n^{\alpha/2+1}\|_{L^6(\Omega)}^{(1-\theta)q}dt \\
  &\le C\|\rho_n^{\alpha/2+1}\|_{L^\infty(0,T;H^1(\Omega))}^{(1-\theta)q}
  \int_0^T\|\rho_n^{\alpha/2+1}\|_{H^2(\Omega)}^{\theta q}dt \le C(\eps),
\end{align*}
and the last step follows from estimate \eqref{3.est2} and the choice $q=10$, which yields $\theta q=2$. By the Gagliardo--Nirenberg inequality with $\theta=3(q-2)/(2q)$,
\begin{align*}
  \|\na\rho_n^{\alpha/2+1}\|_{L^{q}(\Omega_T)}^{q}
  &\le C\int_0^T\|\na\rho_n^{\alpha/2+1}\|_{H^1(\Omega)}^{\theta q}
  \|\na\rho_n^{\alpha/2+1}\|_{L^2(\Omega)}^{(1-\theta)q}dt \\
  &\le C\|\na\rho_n^{\alpha/2+1}
  \|_{L^\infty(0,T;L^2(\Omega))}^{(1-\theta)q}\int_0^T
  \|\na\rho_n^{\alpha/2+1}\|_{H^1(\Omega)}^{\theta q}dt
  \le C(\eps),
\end{align*}
if $\theta q=2$, giving $q=10/3$.
\end{proof}

The limit $n\to\infty$ can be carried out as in \cite[Sec.~2]{VaYu16}, except for the Korteweg contribution. In \cite{VaYu16}, the strong convergence of $(\rho_n)$ and $(\rho_n v_n)$ follows from the Aubin--Lions lemma, relying on the estimate $L^q(0,T;H^{-s}(\Omega))$ for the time derivatives of $\rho_n$ and $\rho_n v_n$ for some $q>1$ and $s\in\N$. This same estimate remains valid in the present setting, since the additional contribution from the Korteweg tensor satisfies $\diver\K(\rho_n)\in L^{10/9}(0,T; W^{-1,3}(\Omega))$. Indeed, let $\psi\in L^{10}(0,T;W^{1,3}(\Omega))$, then
\begin{align*}
  \bigg|&\int_0^T\int_\Omega\diver\mathbb{K}(\rho_n)\cdot\psi 
  dxdt\bigg|
  = \bigg|\kappa\int_0^T\int_\Omega\rho_n\psi\cdot\na
  \bigg(\rho_n^{\alpha/2}\frac{\Delta\rho_n^{\alpha/2+1}}{
  \alpha/2+1}\bigg)dxdt\bigg| \\
  &= \frac{\kappa}{\alpha/2+1}\bigg|\int_0^T\int_\Omega
  \rho_n^{\alpha/2}\Delta\rho_n^{\alpha/2+1}
  (\na\rho_n\cdot\psi + \rho_n\diver\psi)dxdt\bigg| \\
  &= \frac{\kappa}{\alpha/2+1}\bigg|\int_0^T\int_\Omega
  \Delta\rho_n^{\alpha/2+1}\bigg(\frac{\na\rho_n^{\alpha/2+1}}{
  \alpha/2+1}\cdot\psi + \rho_n^{\alpha/2+1}\diver\psi\bigg)
  dxdt\bigg| \\
  &\le C\|\Delta\rho_n^{\alpha/2+1}\|_{L^2(\Omega_T)}
  \big(\|\na\rho_n^{\alpha/2+1}\|_{L^{5/2}(\Omega_T)}
  \|\psi\|_{L^{10}(\Omega_T)}
  + \|\rho_n^{\alpha/2+1}\|_{L^{10}(\Omega_T)}
  \|\diver\psi\|_{L^{5/2}(\Omega_T)}\big) \\
  &\le C\|\psi\|_{L^{10}(0,T;W^{1,3}(\Omega))}.
\end{align*}

We now turn to the convergence of the Korteweg term. First, we observe that the density $\rho_n$ is uniformly positive a.e.\ thanks to the following estimate:
\begin{align*}
  \|\rho^{-1}\|_{L^\infty(\Omega)}
  \le C\big(1+\|\rho^{-1}\|_{L^3(\Omega)}\big)^3
  \big(1+\|\rho\|_{H^k(\Omega)}\big)^2,
\end{align*}
which holds for $\rho\in H^k(\Omega)$ with $k>7/2$ such that $\rho^{-1}\in L^3(\Omega)$ \cite[formula (13)]{BrDe06}. We deduce from the estimate
\begin{align}\label{3.rho}
  \eta^{1/3}\|\rho_n^{-1}\|_{L^\infty(0,T;L^3(\Omega))}
  + \sqrt\delta\|\rho_n\|_{L^\infty(0,T;H^5(\Omega))} \le C(\rho^0,v^0),
\end{align}
which is a consequence of \eqref{eq: 3.ei}, that
\begin{align*}
  \|\rho_n^{-1}\|_{L^\infty(\Omega)}
  \le C(1+C\eta^{-1/3})^3(1+C\delta^{-1/2})^2
  \le C(1+\eta^{-1})(1+\delta^{-1}) =: C_{\delta,\eta}.
\end{align*}

Since $(\na\rho_n)$ is bounded in $L^2(0,T;H^1(\Omega;\R^3))$ and $(\pa_t(\na\rho_n))$ is bounded in $L^1(0,T;$ $H^{-s}(\Omega;\R^3))$ for some $s\in\N$, the Aubin--Lions lemma yields the existence of a subsequence that is not relabeled such that, as $n\to\infty$,
\begin{align*}
  \na\rho_n\to\na\rho\quad\mbox{strongly in }L^2(\Omega_T).
\end{align*}
The uniform positive bound implies that 
\begin{align*}
  \na\rho_n^{\alpha/2+1}\to\na\rho^{\alpha/2+1}
  \quad\mbox{strongly in }L^{2}(0,T;L^{6}(\Omega)).
\end{align*}
 Furthermore, Lemma \ref{lem.n} yields
\begin{align*}
  \rho_n^{\alpha/2+1}\to\rho^{\alpha/2+1}
  &\quad\mbox{strongly in }L^{9}(\Omega_T), \\
  \Delta\rho_n^{\alpha/2+1}\rightharpoonup\Delta\rho^{\alpha/2+1}
  &\quad\mbox{weakly in }L^2(\Omega_T).
\end{align*}
Therefore, for smooth test functions $\psi$,
\begin{align}
  \int_0^T\int_\Omega\diver\mathbb{K}(\rho_n)\cdot\psi dxdt
  &= \frac{\kappa}{\alpha/2+1}\int_0^T\int_\Omega
  \Delta\rho_n^{\alpha/2+1}\bigg(\frac{\na\rho_n^{\alpha/2+1}}{
  \alpha/2+1}\cdot\psi + \rho_n^{\alpha/2+1}\diver\psi\bigg)
  dxdt \nonumber \\
  &\to \frac{\kappa}{\alpha/2+1}\int_0^T\int_\Omega
  \Delta\rho^{\alpha/2+1}\bigg(\frac{\na\rho^{\alpha/2+1}}{
  \alpha/2+1}\cdot\psi + \rho^{\alpha/2+1}\diver\psi\bigg)dxdt.
  \label{3.K}
\end{align}
The limit $n\to\infty$ in the remaining terms can be performed as in \cite[Sec.~2]{VaYu16}. Summarizing, we have shown that the limit $(\rho,v)$ satisfies the following equation:
\begin{align*}
  0 &= \int_\Omega\rho^0v^0\cdot\psi(0)dxdt
  +\int_0^T\int_\Omega \bigg(\rho v\cdot\pa_t\psi +\rho v\otimes v:\na\psi\bigg) dxdt\\
  &\phantom{xx}+ \int_0^T\int_\Omega\bigg( p(\rho)\diver\psi 
  + \nu v\cdot\diver(\rho\mathbb{D}(\psi))- \kappa\rho^{\alpha/2+1}\frac{\Delta\rho^{\alpha/2+1}}{\alpha/2+1}
  \diver(\rho\psi)
  \bigg)dxdt \nonumber \\
  &\phantom{xx}- \int_0^T\int_\Omega\bigg(
    \sigma v\cdot\psi
  + \sigma\rho|v|^2v\cdot\psi + \sigma \K_Q(\rho):\nabla\psi\bigg)dxdt \nonumber \\
  &\phantom{xx}- \int_0^T\int_\Omega\big(
  \eps v\otimes\na\rho:\na\psi + \mu\Delta v\cdot\Delta\psi
  + \eta\rho^{-3}\diver\psi
  + \delta\Delta^3\rho\Delta^3\diver(\rho\psi)\big)dxdt \nonumber 
\end{align*}
for all $\psi\in C^1([0,T];X_n)$ with $\psi(T,\cdot)=0$ for all $n\in\N$. By density, the previous equation also holds for all $\psi\in L^\infty(0,T;H^4(\Omega;\R^3))\cap L^2(0,T;H^5(\Omega;\R^3))$ with $\psi(T,\cdot)=0$. 

It remains to perform the limit $n\to\infty$ in the approximate energy inequality \eqref{eq: 3.ei}. This is done by means of the weakly lower semicontinuity of convex functions. Thus, $(\rho,v)$ also satisfies \eqref{eq: 3.ei}. 


\subsection{BD entropy inequality}

We derive the BD entropy inequality for the constructed solution to the approximate system, providing additional estimates. Because of \eqref{3.rho}, the density $\rho$ is uniformly bounded away from zero, which allows us to use $\na\log\rho$ etc.\ as a test function in the approximate mass equation \eqref{eq: 2.mass}. We recall the definition of the BD entropy:
\begin{align*}
  E_{BD}(\rho,v) = \frac12\int_{\Omega}\rho|v+\nu\na\log\rho|^2 dx.
\end{align*}

\begin{lemma}[BD entropy inequality]\label{lem.BD}
There exist constants $C_1>0$ independent of $(\sigma,\delta,\eta$, $\eps,\mu)$ and $C_2(\sigma, \delta,\eta)>0$ independent of $(\eps,\mu)$ such that for $0<t<T$,
\begin{align*}
  E_{BD}((\rho,v)(&t)) + \nu \sigma\int_\Omega(\log\rho(t,x))_-dx
  + \frac{4\nu}{\gamma}\int_0^t\int_\Omega
  |\na\rho^{\gamma/2}|^2 dxd\tau \\
  &\phantom{xx}+ \frac{\nu}{4}\int_0^t\int_\Omega
  \rho|\na v-\na v^T|^2 dxd\tau
  + \eps\nu^2\int_0^t\int_\Omega\frac{(\Delta\rho)^2}{\rho}dxd\tau \\
  &\phantom{xx}+ \delta\nu\int_0^t\int_\Omega(\Delta^3\rho)^2 dxd\tau
  + \frac43\eta\nu\int_0^t\int_\Omega|\na\rho^{-3/2}|^2 dxd\tau \\
  &\phantom{xx}+ C_0\nu\kappa\int_0^t\int_\Omega\big(
  \mathrm{1}_{\alpha<0}|\na\rho^{\alpha/4+1/2}|^4
  + |\Delta\rho^{\alpha/2+1}|^2\big)dxd\tau\\
  &\phantom{xx}+ \widetilde{C}_0\nu\sigma\int_0^t\int_\Omega\big(
  |\na\rho^{1/4}|^4
  + |\Delta\sqrt{\rho}|^2\big)dxd\tau\\
  &\le C_1(1+\sigma) + \bigg(\eps+\sqrt\mu+\frac{\eps}{\sqrt\mu}\bigg)
  C_2(\sigma,\delta,\eta),
\end{align*}
recalling that $z_-=\max\{0,z\}$. 
\end{lemma}

\begin{proof}
    The BD entropy {\em equality} for sufficiently smooth solutions has been derived in several papers; see, e.g., \cite{BrDe06,JuPh26,Zat12}. We omit the technical calculation and just state the result, including the Korteweg term:
\begin{align}\label{eq: 3.bd}
    \frac12&\int_\Omega\rho|v+\nu\na\log\rho|^2 dx\Big|_0^t
    + \nu\int_0^t\int_\Omega\rho|\mathbb{D}(v)|^2 dxd\tau 
    + \frac{\nu}{4}\int_0^t\int_\Omega\rho|\na v-\na v^T|^2 dxd\tau \\
    &\phantom{xx}+ \frac{4\nu}{\gamma}\int_0^t\int_\Omega
    |\na\rho^{\gamma/2}|^2 dxd\tau
    + \eps\nu^2\int_0^t\int_\Omega\frac{(\Delta\rho)^2}{\rho}dxd\tau
    + \delta\nu\int_0^t\int_\Omega(\Delta^3\rho)^2 dxd\tau \nonumber \\
    &\phantom{xx}+ \frac43\eta\nu\int_0^t\int_\Omega
    |\na\rho^{-3/2}|^2 dxd\tau
    - 2\nu\sigma\int_0^t\int_\Omega\na\rho\cdot\na\bigg(
    \frac{\Delta\sqrt{\rho}}{\sqrt{\rho}}
    \bigg)dxd\tau \nonumber \\
    &\phantom{xx}- \nu\kappa\int_0^t\int_\Omega\na\rho\cdot\na\bigg(
    \rho^{\alpha/2}\frac{\Delta\rho^{\alpha/2+1}}{\alpha/2+1}
    \bigg)dxd\tau \nonumber\\
    &= \frac{\eps}{2}\nu^2\int_0^t\int_\Omega
    \Delta\rho|\na\log\rho|^2 dxd\tau
    - \nu \sigma\int_0^t\int_\Omega v\cdot\na\log\rho dxd\tau \nonumber \\
    &\phantom{xx}- \nu \sigma\int_0^t\int_\Omega |v|^2 v\cdot\na\rho dxd\tau
    - \nu\mu\int_0^t\int_\Omega\Delta^2 v\cdot\na\log\rho dxd\tau
    \nonumber \\
    &\phantom{xx}+ \eps\nu\int_0^t\int_\Omega
    \na v:(\na\rho\otimes\na\log\rho) dxd\tau
    - \eps\nu\int_0^t\int_\Omega\diver v\Delta\rho dxd\tau
    + \frac12\int_\Omega\rho|v|^2 dx\Big|_0^t \nonumber \\
    &=: I_1 + \cdots + I_7. \nonumber 
\end{align}
In view of inequality \eqref{2.ineq}, the last two terms on the left-hand side are nonnegative:
\begin{align*}
  - \nu\kappa&\int_0^t\int_\Omega\na\rho\cdot\na\bigg(
  \rho^{\alpha/2}\frac{\Delta\rho^{\alpha/2+1}}{\alpha/2+1}\bigg)dxds \\
  &\ge C_0\nu\kappa\int_0^t\int_\Omega
  \big(\mathrm{1}_{\alpha<0}|\na\rho^{\alpha/4+1/2}|^4
  + |\Delta\rho^{\alpha/2+1}|^2\big)dxds, \\
  - 2\nu\sigma&\int_0^t\int_\Omega\na\rho\cdot\na\bigg(
 \frac{\Delta\sqrt{\rho}}{\sqrt\rho}\bigg)dxds \ge \widetilde{C}_0\nu\sigma\int_0^t\int_\Omega
  \big(|\na\rho^{1/4}|^2
  + |\Delta\sqrt{\rho}|^2\big)dxds.
\end{align*}

In the following, we estimate the terms $I_1,\ldots,I_7$. The energy inequality \eqref{eq: 3.ei} shows that $I_7\le C(\rho^0,v^0)$. The terms $I_1$, $I_4$, $I_5$, and $I_6$ correspond to $J_1$, $J_2$, $J_3$, and $J_4$, respectively, in \cite[Sec.~3.3]{JuPh26}, leading to
\begin{align*}
  I_1 &\le C\eps\nu^2\|\Delta\rho\|_{L^2(\Omega_T)}
  \|\rho^{-1}\|_{L^\infty(\Omega_T)}^2\|\na\rho\|_{L^2(\Omega_T)}^2
  = C(\delta,\eta)\eps,\\
  I_4 &= -\nu\mu\int_0^t\int_\Omega\Delta v\cdot\na\Delta\log\rho dxds
  \le \nu\mu\|\Delta v\|_{L^2(\Omega_T)}
  \|\na\Delta\log\rho\|_{L^2(\Omega_T)} 
  \le C(\delta,\eta)\sqrt\mu,\\
  I_5 &\le C\eps\|\sqrt\rho\mathbb{D}(v)\|_{L^2(\Omega_T)}
  \|\na\sqrt[4]{\rho}\|_{L^4(\Omega_T)}^2
  \le C(\sigma)\eps^{1/2},\\
  I_6 &\le C\eps\|\rho\|_{L^2(0,T;H^1(\Omega))}
  \|\Delta v\|_{L^2(\Omega_T)}\le C(\sigma, \delta)\eps \mu^{-1/2}.
\end{align*}
The terms $I_2$ and $I_3$ correspond to $K_3$ and $K_4$ in \cite[Sec.~3.3]{JuPh26}, respectively, where it is shown that
\begin{align*}
  I_2 &\le C \sigma + C(\delta,\eta)\eps - \nu \sigma\int_\Omega
  (\log\rho(t,x))_- dx, \\
  I_3 &\le \nu \sigma\int_0^t\int_\Omega\bigg(\frac34\rho|v|^4
  + \frac52\rho|\mathbb{D}(v)|^2\bigg)dxds
  \le C(1+\sigma).
\end{align*}
Collecting these estimates in \eqref{eq: 3.bd} finishes the proof.
\end{proof}

The bounds of Lemma \ref{lem.n} depend on $\eps$. This can be improved by means of the BD entropy inequality \eqref{eq: 3.ei}. Indeed, the $\eps$-dependence comes from the fact that the integral
\begin{align*}
  \int_0^t\int_\Omega\big(\mathrm{1}_{\alpha<0}|\na\rho^{\alpha/4+1/2}|^2
  + |\Delta\rho^{\alpha/2+1}|^2\big)dxds
\end{align*}
is bounded from above by $C/\eps$. Lemma \ref{lem.BD} shows that this integral is bounded by 
\begin{align*}
  C_1(1+\sigma) + \bigg(\eps+\sqrt\mu+\frac{\eps}{\sqrt\mu}\bigg)
  C_2(\sigma,\delta,\eta),
\end{align*}
which becomes independent of $(\delta,\eta)$ if $(\eps,\mu)\to 0$ with $\eps/\sqrt\mu\to 0$ . This limit is performed in the following section.


\subsection{Limit $(\delta,\eta,\eps,\mu)\to 0$}\label{sec.lim}

The results of the previous section allow us to perform the limit $(\eps,\mu)\to0$ under the condition $\eps/\sqrt{\mu}\to0$. As in \cite[Sec.~3.2]{VaYu16}, this removes the dependence of the right-hand side of the BD entropy inequality on $(\delta,\eta)$ and yields estimates that are uniform with respect to these parameters. We can therefore pass to the limit $(\delta,\eta)\to0$ similarly as in \cite[Sec.~3.3]{VaYu16}. As before, the only remaining difficulty is the convergence of the Korteweg term. For the limit $(\varepsilon,\mu)\to0$, the argument follows the same strategy as in the Faedo--Galerkin limit $n\to\infty$. In the case $\delta=\eta$, the limit $\delta\to0$ can again be carried out by means of the Aubin--Lions theorem, which gives the strong convergence $\rho_\delta\to\rho$ in $L^2(\Omega_T)$. This and the following convergences hold up to subsequences. Consequently, Lemma~\ref{lem.n} yields
\begin{align*}
    \rho_\delta^{\alpha/2+1} \rightarrow \rho^{\alpha/2+1}\quad\mbox{strongly in }L^p(\Omega_T)\mbox{ for all }p<10.
\end{align*}
Moreover, by Lemma~\ref{lem.BD}, 
\begin{align*}
    \Delta \rho_\delta^{\alpha/2+1} \rightarrow \Delta \rho^{\alpha/2+1} \quad\mbox{weakly in }L^2(\Omega_T).
\end{align*}
The convergence of $\nabla \rho^{\alpha/2+1}_\delta$ requires a different argument, since the positive lower bound of the density is no longer available. At this stage, the Aubin--Lions theorem could still be applied, since we can bound $\partial_t \rho^{\alpha/2+1}_\delta$. However, because this estimate is lost in the subsequent limit $\sigma\to0$, we use a different argument here.

\begin{lemma}\label{lem.strConvGradient}
It holds that, as $\delta \to 0$,
\begin{align*}
    \nabla \rho_\delta^{\alpha/2+1} \rightarrow\nabla\rho^{\alpha/2+1}  \quad \mbox{strongly in } L^p(\Omega_T)\ \mbox{for all }  p<10/3.
\end{align*}
\end{lemma}

\begin{proof}
    The argument follows \cite[Lem.~5.2]{ABS25}. Set $\theta = \alpha/2+1$. Since $\nabla \rho^{\theta} = 2\rho^{\theta/2} \nabla \rho^{\theta/2}$, it is sufficient to prove the strong convergence of $\rho_\delta^{\theta/2}\nabla\rho_\delta^{\theta/2}$ in $L^2(\Omega_T)$ and then to improve the space of convergence by the uniform estimate $\|\nabla \rho^{\alpha/2+1}\|_{L^{10/3}(\dom_T)}\leq C$ from Lemma~\ref{lem.n}.
    We deduce from the strong convergence of $\rho_\delta$ as well as Lemma~\ref{lem.n} that
    \begin{align*}
        \rho_{\delta}^{\theta/2} \rightarrow \rho^{\theta/2} \quad\mbox{strongly in } L^{4}(\dom_T),\\
        \nabla \rho_\delta^{\theta/2}\rightarrow  \nabla \rho^{\theta/2}\quad\mbox{weakly in }L^4(\Omega_T).
    \end{align*}
    This shows that $ \rho_\delta^{\theta/2} \nabla \rho_\delta^{\theta/2}\rightharpoonup \rho^{\theta/2} \nabla \rho^{\theta/2}$ weakly in $L^2(\Omega_T)$. 
    To prove strong convergence, we compute
    \begin{align*}
        \|\rho_{\delta}^{\theta/2} &\nabla \rho_\delta^{\theta/2} - \rho^{\theta/2} \nabla \rho^{\theta/2}\|_{L^2(\Omega_T)}^2
        = \int_0^T\int_\dom \rho_\delta^{\theta} |\nabla\rho_\delta^{\theta/2}|^2dxdt \\
        &\phantom{xx}- 2\int_0^T\int_\dom (\rho_\delta^{\theta/2} \nabla \rho_{\delta}^{\theta/2}) \cdot (\rho^{\theta/2} \nabla \rho^{\theta/2}) dxdt 
        + \int_0^T\int_\dom \rho^{\theta} |\nabla\rho^{\theta/2}|^2 dxdt\\  
        &= -\frac{1}{4}\int_{0}^T\int_{\dom} \rho_{\delta}^{\theta}\Delta \rho_{\delta}^{\theta}dxdt - 2\int_0^T\int_\dom (\rho_\delta^{\theta/2} \nabla \rho_{\delta}^{\theta/2}) \cdot (\rho^{\theta/2} \nabla \rho^{\theta/2}) dxdt \\
        &\phantom{xx}+ \int_0^T\int_\dom \rho^{\theta} |\nabla\rho^{\theta/2}|^2 dxdt
    \end{align*}
    where we used in the last step the identity
    \begin{align}\label{eq: identity conv rho alpha/2+1}
      \int_{0}^T\int_{\dom}\rho_\delta^{\theta}|\nabla \rho_\delta^{\theta/2}|^2dxdt
      = \frac14\int_0^T\int_\dom|\nabla\rho_\delta^\theta|^2 dxdt
      = -\frac{1}{4}\int_{0}^T\int_{\dom} \rho_\delta^{\theta}\Delta \rho_\delta^{\theta}dxdt.
    \end{align}
    The limit $\delta \to 0$ yields that
    \begin{align*}
        \lim_{\delta\to 0}&\|\rho_{\delta}^{\theta/2} \nabla \rho_\delta^{\theta/2} - \rho^{\theta/2} \nabla \rho^{\theta/2}\|_{L^2(\Omega_T)}^2
        = -\frac{1}{4}\int_{0}^T\int_{\dom} \rho^{\theta}\Delta \rho^{\theta}dxdt \\
        &\phantom{xx}- 2\int_0^T\int_\dom (\rho^{\theta/2} \nabla \rho^{\theta/2}) \cdot (\rho^{\theta/2} \nabla \rho^{\theta/2}) dxdt + \int_0^T\int_\dom \rho^{\theta} |\nabla\rho^{\theta/2}|^2 dxdt\\
        &= -\frac{1}{4}\int_{0}^T\int_{\dom} \rho^{\theta}\Delta \rho^{\theta}dxdt -\int_0^T\int_\dom \rho^{\theta} |\nabla\rho^{\theta/2}|^2 dxdt =0
    \end{align*}
    where the last equality follows again from \eqref{eq: identity conv rho alpha/2+1}.
\end{proof}

The previous lemma and the arguments of \cite[Sec.~3.3]{VaYu16} show that the limit $(\delta,\eta)\to 0$ in the Korteweg term can be carried out exactly as in \eqref{3.K}. The final step of the proof is to extend the result to initial data with lower regularity. This is achieved by a standard approximation argument: We construct a sequence of smooth functions that approximate the initial data. This finishes the construction of a weak solution to the regularized Navier--Stokes--Korteweg system \eqref{eq: 1a.mass}--\eqref{eq: 1a.momentum} in the sense of Definition \ref{def.weakregul}.


\section{From weak to renormalized weak solutions}\label{sec.wkTOr}

In this section, we show that any weak solution to the regularized Navier--Stokes-Korteweg system is a renormalized weak solution in the sense of Definition \ref{def.renormregul}. To this end, let $(\sqrt{\rho}, \sqrt{\rho}v)$ be a weak solution to \eqref{eq: 1a.mass}--\eqref{eq: 1a.momentum} with $\sigma>0$ on $[0,T]$ with initial conditions $(\sqrt{\rho^0}$, $\sqrt{\rho^0}v^0)$ satisfying \eqref{eq: 1.regIC} and such that $\sqrt{\rho^0}$ is uniformly bounded away from zero. 

The derivation of the renormalized formulation is based on the formal choice of the test function $\varphi'(v)\zeta(\rho)\theta$ in the momentum equation, with $\varphi$, $\zeta$ and $\theta$ given in Definition \ref{def.renormregul}. Since this is not justified at the level of weak solutions, due to their limited regularity (see \cite[Sec.~3]{LaVa18} for details), we regularize the test function. The proof follows the strategy of \cite[Theorem~3.2]{AnSp22}. Since that work considers a different Korteweg term, the available regularity differs from ours, so the proof has to be verified step by step. Up to the introduction of the density truncation $\zeta(\rho)$, however, we can follow \cite[Sec.~3]{LaVa18}, whose regularity coincides with ours, except for the additional regularity provided by our Korteweg term.

To this end, let $\eta$ be a normalized, smooth, nonnegative, even function with compact support in the unit space-time ball, and define
\begin{align*}
  \eta_\eps(t,x) = \eps^{-d-1}\eta(t/\eps,x/\eps), \quad
  \eps>0,\ (t,x)\in(0,T)\times\R^d.
\end{align*}
Since the approximation parameters $\eta$ and $\eps$ have already been sent to zero in the previous section, we reuse this notation here without causing any ambiguity. The mollification of a function $g$ is defined by
\begin{align}\label{5.molli}
  \overline{g}^\eps(t,x) = (\eta_\eps * g)(t,x), \quad
  (t,x)\in(0,T)\times\R^d.
\end{align}
The mollification satisfies some symmetry and commutator properties detailed in \cite[Sec.~3.1]{LaVa18}. We introduce the continuous cutoff function
\begin{align}\label{5.cutoff}
  \phi_m(y) = \begin{cases}
  0 &\mbox{for }0\le y\le 1/(2m), \\
  2my-1 &\mbox{for }1/(2m)\le y\le 1/m, \\
  1 &\mbox{for }1/m\le y\le m, \\
  2-y/m &\mbox{for }m\le y\le 2m, \\
  0 &\mbox{for }y\ge 2m,
  \end{cases}
\end{align}
and define the density-cutoff velocity $\widetilde{v}_m:=\phi_m(\rho)v$. 
Before we start, we summarize the estimates used throughout this section.
\begin{lemma}\label{lem.reg}
Let $-1 \le \alpha \le 0$. There exists a constant $C$ independent of $\sigma\in(0,1)$ such that
    \begin{align*}
    \|\rho^{\alpha/2+1}\|_{L^{10}(\Omega_T)}
    + \|\rho\|_{L^{q_1}(\Omega_T)} &\le C, \\
    \|\na\rho^{\alpha/2+1}\|_{L^{10/3}(\Omega_T)}+\|\na\rho\|_{L^{q_2}(\Omega_T)}
     + \|\nabla \sqrt{\rho}\|_{L^{q_3}(\Omega_T)}
    &\le C,\\
    \|\rho v\|_{L^{q_4}(0,T; L^2(\dom_T))} + \|\nabla (\rho v)\|_{L^{q_5}(\dom_T)}+ \|\partial_t \rho\|_{L^{q_5}(\dom_T)}&\leq C,\\
    \| \rho^{\gamma/2}\|_{L^{10/3}(\dom_T)} &\leq C
    \end{align*}
    holds for $q_1=5(\alpha+2)\ge 5$, $q_2=5(\alpha+2)/(\alpha+3) \ge 5/2$, $q_3 = 10(\alpha+2)/(2\alpha+5))\ge 10/3$, $q_4 = 2(\alpha+2)\geq 2$, and $q_5= 10(\alpha+2)/(7\alpha+15)\geq 5/4$. 
\end{lemma}

\begin{proof}
The estimates for $\rho^{\alpha/2+1}$ and $\nabla \rho^{\alpha/2+1}$ are obtained exactly as in Lemma~\ref{lem.n}. Thus, $\rho$ is bounded in $L^{q_1}(\Omega_T)$ with $q_1=10(\alpha/2+1)=5(\alpha+2)\ge 5$. Furthermore,
    \begin{align*}
    \|\na\rho\|_{L^{q}(\Omega_T)}
    &= C(\alpha)\|\rho^{-\alpha/2}
    \na\rho^{\alpha/2+1}\|_{L^q(\Omega_T)} \\
    &\le C\|\rho^{-\alpha/2}\|_{L^{-2q_1/\alpha}(\Omega_T)}
    \|\na\rho^{\alpha/2+1}\|_{L^{10/3}(\Omega_T)} \le C,
    \end{align*}
    where $1/q = -\alpha/(2q_1) + 3/10$ and hence $q_2:=q=5(\alpha+2)/(\alpha+3)$. We deduce from H\"older's inequality that
    \begin{align*}
        \|\na\rho^{1/2}\|_{L^{q}(\Omega_T)}
        &= C(\alpha)\|\rho^{-\alpha/4}
        \na\rho^{\alpha/4+1/2}\|_{L^{q}(\Omega_T)} \\
        &\le C\|\rho^{-\alpha/4}\|_{L^{-4q_1/\alpha}(\Omega_T)}
        \|\na\rho^{\alpha/4+1/2}\|_{L^4(\Omega_T)} \le C,
    \end{align*}
    where $1/q = -\alpha/(4q_1)+1/4$, which gives $q_3:=q=10(\alpha+2)/(2\alpha+5)\ge 10/3$. The Sobolev embedding $\rho^{\alpha/2+1}\in L^2(0,T; H^2(\dom))\hookrightarrow L^2(0,T; L^{\infty}(\dom))$ implies that $\sqrt{\rho} \in L^{q_4}(0,T; L^\infty(\dom))$ with $q_4= 2(\alpha+2)$. Then it follows from $\sqrt{\rho}v\in L^{\infty}([0,T]; L^2(\dom))$ that $\rho v$ is uniformly bounded in $L^{q_4}(0,T; L^2(\dom_T))$. Next, by definition \eqref{eq: 1.Tnu} of the tensor $\T_\nu$, we find that
    \begin{align*}
        \|\nabla (\rho v)\|_{L^q(\dom_T)} \leq C (\|\sqrt{\rho}\|_{L^{2q_1}(\dom_T)} \|\T_{\nu}\|_{L^2(\Omega_T)} + \|\sqrt{\rho}v\|_{L^2(\dom_T)} \| \nabla \sqrt{\rho}\|_{L^{q_3}(\Omega_T)})
    \end{align*}
    where $1/q= 1/2 +1/q_3$, which yields that $q_5 := q= 10(\alpha+2)/(7\alpha+15)\geq 5/4$. The estimate of $\partial_t \rho$ follows directly from the mass equation. The estimate for $\rho^{\gamma/2}$ is obtained by the same argument used to derive the estimate for $\nabla \rho^{\alpha/2+1}$.
\end{proof}

The estimates of the previous lemma imply the following bounds:
\begin{align*}
        \|\nabla \rho^{\alpha/2+1}\Delta \rho^{\alpha/2+1}\|_{L^{5/4}(\Omega_T)} + \|\rho^{\alpha/2+1} \Delta \rho^{\alpha/2+1}\|_{L^{5/3}(\dom_T)} 
        + \|\rho^{\gamma/2}\nabla\rho^{\gamma/2}\|_{L^{5/4}(\dom_T)}\leq C.
    \end{align*}
We will use the following regularity property, whose proof is given in \cite[Lemma 3.3]{LaVa18}.
\begin{lemma} There exists a constant $C(\sigma)$ such that
    \begin{align*}
        \|\rho v \|_{L^{5/2}(\dom_T)} + \|\sqrt{\rho}\S_{\nu} + \K_{Q}\|_{L^{5/3}(\dom_T)} + \|\rho |v|^2 v\|_{L^{5/4}(\dom_T)}\leq C(\sigma),\\
        \| \partial_t \phi_m(\rho)\|_{L^2(\dom_T)} + \|\nabla \phi_m(\rho)\|_{L^4(\dom_T)}\leq C(\sigma).
    \end{align*}
\end{lemma}


\subsection{Incorporating the truncation of the velocity}

We first derive a momentum equation involving $\widetilde{v}_m = \phi_m(\rho)v$, recalling definition \eqref{5.cutoff} of the cutoff function $\phi_m$. This is achieved by testing the momentum equation \eqref{eq: 1.wmomentum} with $\overline{\phi_m(\rho)\psi}^{\eps}$ (see definition \eqref{5.molli}) and passing to the limit $\eps\to0$, following exactly the argument in \cite{LaVa18} and in particular using \cite[Lemmas~3.1 and~3.2]{LaVa18}. The only terms requiring additional attention are those arising from the Korteweg contributions. For the first term, we have
\begin{align*}
    &\int_0^T \int_{\dom}  \rho^{\alpha/2+1}\frac{\Delta \rho^{\alpha/2 +1}}{\alpha/2+1}\diver(\overline{\psi\phi_m(\rho)}^{\eps})dxdt\\
    &= \int_0^T \int_{\dom}  \overline{\rho^{\alpha/2+1}\frac{\Delta \rho^{\alpha/2+1}}{\alpha/2+1}}^{\eps} \big(\nabla \phi_m(\rho)\cdot \psi + \phi_m(\rho)\diver\psi\big)dxd\tau\\
    &\rightarrow \int_0^T \int_{\dom} \rho^{\alpha/2+1}\frac{\Delta \rho^{\alpha/2+1}}{\alpha/2+1} \big(\nabla \phi_m(\rho) \cdot\psi + \phi_m(\rho)\diver\psi\big)dxd\tau.
\end{align*}
Similarly, for the second contribution of the Korteweg term, we obtain
\begin{align*}
    \int_0^T \int_{\dom}  \frac{\Delta \rho^{\alpha/2+1}}{(\alpha/2+1)^2}\nabla \rho^{\alpha/2+1} \cdot \overline{\phi_m(\rho)\psi}^{\eps} dxdt \rightarrow \int_0^T \int_{\dom} \frac{\Delta \rho^{\alpha/2+1}}{(\alpha/2+1)^2}\nabla \rho^{\alpha/2+1}  \cdot \psi\phi_m(\rho) dxdt.
\end{align*}
Thus, the following equation holds:
\begin{align*}
    \int_{\dom}&\rho^0 v^0\cdot \psi(0,\cdot) \phi_m(\rho^0)dx 
    + \int_0^T \int_{\dom}\big(\rho \widetilde{v}_m \cdot \partial_t \psi + (\rho v \otimes \widetilde{v}_m): \nabla \psi\big) dx dt \\ 
    &-\int_0^T \int_{\dom}\big(M_1 \phi_m(\rho): \nabla \psi + M_2 \cdot \psi\big) dxdt - \int_0^T\int_{\dom} M_1 : (\nabla \phi_m(\rho)\otimes \psi)dxdt\nonumber\\
    &-\int_0^T \int_{\dom}\kappa\bigg(\frac{\Delta \rho^{\alpha/2+1}}{(\alpha/2+1)^2}\nabla\rho^{\alpha/2+1} \cdot \psi \phi_m(\rho) + \rho^{\alpha/2+1} \frac{\Delta \rho^{\alpha/2+1}}{\alpha/2+1}\phi_m(\rho)\diver\psi \nonumber \\
    &\phantom{xx}+  \rho^{\alpha/2+1} \frac{\Delta \rho^{\alpha/2+1}}{\alpha/2+1}\nabla \phi_m(\rho)\cdot \psi\bigg) dxdt
    -\int_0^T \int_{\dom} \sqrt{\frac{\rho}{\nu}}  \tr(\T_{\nu}) \phi_m'(\rho) \rho v \psi dxdt = 0, \nonumber
\end{align*}
where we have set
\begin{align}\label{4.M}
    M_1 = \sqrt{\nu \rho} \S_{\nu} -\sigma\K_{Q},\quad
    M_2= \sigma v+ \sigma \rho v |v|^2 + 2 \rho^{\gamma/2}\nabla \rho^{\gamma/2}.
\end{align}


\subsection{Incorporating the renormalization of the velocity}

We test the momentum equation with the test function $\overline{\theta\varphi'\left(\overline{\widetilde{v}_m}^{\eps}\right)}^\eps$, where $\theta \in C_0^{\infty}([0,T)\times\Omega)$. Again, all terms except for the new Korteweg contribution can be treated exactly as in \cite[Sec.~3.3]{LaVa18}. It follows from our uniform bounds that the new Korteweg contribution converges:
\begin{align*}
    \int_0^T \int_{\dom}&\frac{\Delta \rho^{\alpha/2+1}}{(\alpha/2+1)^2}  \nabla\rho^{\alpha/2+1}\cdot \overline{\theta\varphi'\left(\overline{\widetilde{v}_m}^{\eps}\right)}^\eps dxdt \\
    &\rightarrow \int_0^T \int_{\dom} \frac{\Delta \rho^{\alpha/2+1}}{(\alpha/2+1)^2}\nabla\rho^{\alpha/2+1} \cdot \theta\varphi'\left(\widetilde{v}_m\right)dxdt, \\
    \int_0^T \int_{\dom}&\rho^{\alpha/2+1} \frac{\Delta \rho^{\alpha/2+1}}{\alpha/2+1}\nabla \phi_m(\rho)\cdot \overline{\theta\varphi'\left(\overline{\widetilde{v}_m}^{\eps}\right)}^\eps dxdt \\
    &\rightarrow\int_0^T \int_{\dom} \rho^{\alpha/2+1} \frac{\Delta \rho^{\alpha/2+1}}{\alpha/2+1}\nabla \phi_m(\rho)\cdot (\theta\varphi'(\widetilde{v}_m)) dxdt, \\
    \int_0^T \int_{\dom} & \rho^{\alpha/2+1} \frac{\Delta \rho^{\alpha/2+1}}{\alpha/2+1}\phi_m(\rho)\diver\bigg(\overline{\theta\varphi'\left(\overline{\widetilde{v}_m}^{\eps}\right)}^\eps\bigg)dxdt\\
    &=\int_0^T \int_{\dom} \overline{ \rho^{\alpha/2+1} \frac{\Delta \rho^{\alpha/2+1}}{\alpha/2+1}\phi_m(\rho)}^\eps\Big(\nabla \theta \cdot \varphi'\big(\overline{\widetilde{v}_m}^{\eps}\big) + \theta\varphi''\big(\overline{\widetilde{v}_m}^\eps\big):\nabla \overline{\widetilde{v}_m}^\eps \Big)dxdt\\
    &\rightarrow\int_0^T \int_{\dom} \rho^{\alpha/2+1} \frac{\Delta \rho^{\alpha/2+1}}{\alpha/2+1}\phi_m(\rho)\big(\nabla \theta \cdot \varphi'(\widetilde{v}_m) + \theta\varphi''(\widetilde{v}_m^\eps):\nabla \widetilde{v}_m \big)dxdt.
\end{align*}
Hence, recalling $\widetilde{v}_m=\phi_m(\rho)v$ and definition \eqref{4.M} of $M_1$ and $M_2$, the following equation is satisfied in the limit $\eps\to 0$:
\begin{align}
    \int_{\dom}&\rho^0\phi_m(\rho^0) v^0\cdot \varphi'\big(\phi_m(\rho^0)v^0\big)\theta(0,\cdot) dx 
    + \int_0^T \int_{\dom}\rho \varphi(\widetilde{v}_m)
    (\partial_t \theta + v \cdot \nabla \theta)  dx dt 
    \label{eq: 4.2.rmomentum}\\
    &-\int_0^T \int_{\dom}\big\{\phi_m(\rho)M_1 : (\varphi'(\widetilde{v}_m)\otimes \nabla \theta) + \phi_m(\rho)M_1 : \theta \nabla \varphi'(\widetilde{v}_m) 
    \nonumber \\
    &+ M_1 : (\nabla \phi_m(\rho) \otimes \varphi'(\widetilde{v}_m))\theta\big\}
    dxdt - \int_0^T\int_{\dom} \phi_m(\rho)M_2 \cdot \varphi'(\widetilde{v}_m) \theta dx dt \nonumber \\
    &- \int_0^T \int_{\dom} \kappa \rho^{\alpha/2+1} \frac{\Delta \rho^{\alpha/2+1}}{\alpha/2+1}\big\{\phi_m'(\rho)\nabla \rho \cdot \varphi'(\widetilde{v}_m) + \phi_m(\rho) \varphi''(\widetilde{v}_m) : \nabla\widetilde{v}_m\big\}\theta dxdt\nonumber\\
    &- \int_0^T \int_{\dom} \kappa \rho^{\alpha/2+1} \frac{\Delta \rho^{\alpha/2+1}}{\alpha/2+1} \phi_m(\rho)  \varphi'(\widetilde{v}_m)\cdot \nabla \theta dxdt\nonumber\\
    &-\int_0^T \int_{\dom} \kappa \frac{\Delta \rho^{\alpha/2+1}}{(\alpha/2+1)^2} \phi_m(\rho) \nabla \rho^{\alpha/2+1} \cdot\varphi'(\widetilde{v}_m)\theta dxdt \nonumber \\
    &- \int_0^T\int_{\dom} \sqrt{\frac{\rho}{\nu}} \tr(\T_{\nu}) \phi_m'(\rho) \rho v\cdot \varphi'(\widetilde{v}_m) \theta dxdt =0.\nonumber
\end{align}


\subsection{Incorporating the truncation of the density}

We replace the test function $\theta$ in equation \eqref{eq: 4.2.rmomentum} by $\zeta(\overline{\rho}^{\eps})\theta$ and pass to the limit $\eps \to 0$. The term involving the time derivative of the test function becomes
\begin{align*}
    \int_0^T \int_{\dom}& \rho \varphi(\widetilde{v}_m)\zeta(\overline{\rho}^{\eps}) \partial_t\theta dxdt + \int_0^T \int_{\dom} \rho \varphi(\widetilde{v}_m)\zeta'(\overline{\rho}^{\eps}) \theta \partial_t \overline{\rho}^{\eps}dxdt\\
    &\rightarrow\int_0^T \int_{\dom} \rho \varphi(\widetilde{v}_m)\zeta(\rho) \partial_t\theta dxdt + \int_0^T \int_{\dom} \rho \varphi(\widetilde{v}_m)\zeta'(\rho) \theta \partial_t \rho dxdt.
\end{align*}
The terms that do not involve derivatives of the test function are treated analogously, as the mollified test function is uniformly bounded in $L^{\infty}(\dom_T)$. For terms involving the spatial derivative of the test function, we apply the product rule
\begin{align*}
    \nabla(\zeta(\overline{\rho}^{\eps})\theta) = \zeta (\overline{\rho}^{\eps})\nabla \theta + \theta\zeta'(\overline{\rho}^{\eps})\nabla \overline{\rho}^{\eps}.
\end{align*}
The first contribution belongs to $L^{\infty}(\Omega_T)$. Consequently, passing to the limit in this term, when employed as a test function, is immediate. For the second term, the bound $\|\zeta'(\rho)\nabla \overline{\rho}^{\eps}\|_{L^4(\Omega_T)}\leq C$ is sufficient to justify the passage to the limit. Altogether, we conclude from \eqref{eq: 4.2.rmomentum} that
\begin{align}
  \int_{\dom}&\rho^0\phi_m(\rho^0) v^0\cdot \varphi'\big(\phi_m(\rho^0)v^0\big)\zeta(\rho^0)\theta(0,\cdot)dx  
  \label{eq: 4.3.rmomentum} \\
  &+ \int_0^T \int_{\dom}\rho \varphi(\widetilde{v}_m) \zeta(\rho)
  (\partial_t \theta + v \cdot\zeta(\rho) \nabla \theta) dx dt \nonumber \\
  & + \int_0^T \int_{\dom}\rho \varphi(\widetilde{v}_m)\theta\zeta'(\rho) 
  (\partial_t\rho + v \cdot\nabla \rho\theta)
  dx dt\nonumber\\
  &-\int_0^T \int_{\dom}\phi_m(\rho)\big\{M_1 : (\varphi'(\widetilde{v}_m)\otimes \nabla \theta)\zeta(\rho) 
  + M_1 : (\varphi'(\widetilde{v}_m)\otimes \nabla \rho) \theta \zeta'(\rho)\big\}dxdt \nonumber\\
  &- \int_0^T \int_{\dom}\big\{\phi_m(\rho)M_1 : \nabla \varphi'(\widetilde{v}_m)+ M_1 : (\nabla \phi_m(\rho) \otimes \varphi'(\widetilde{v}_m))\big\}\zeta(\rho)\theta dxdt\nonumber\\
  & -\int_0^T\int_{\dom}\phi_m(\rho) M_2\cdot \varphi'(\widetilde{v}_m) \zeta(\rho)\theta dx dt \nonumber \\
  &- \int_0^T \int_{\dom} \kappa \rho^{\alpha/2+1} \frac{\Delta \rho^{\alpha/2+1}}{\alpha/2+1}\big\{\phi_m'(\rho)\nabla \rho \cdot \varphi'(\widetilde{v}_m) + \phi_m(\rho)\varphi''(\widetilde{v}_m) : \nabla\widetilde{v}_m\big\}\zeta(\rho) \theta dxdt\nonumber\\
  &- \int_0^T \int_{\dom} \kappa \rho^{\alpha/2+1} \frac{\Delta \rho^{\alpha/2+1}}{\alpha/2+1} \phi_m(\rho) \big\{\varphi'(\widetilde{v}_m)\cdot\nabla \theta\zeta(\rho) + \varphi'(\widetilde{v}_m)\cdot \nabla \rho\theta\zeta'(\rho)\big\} dxdt \nonumber \\
  &-\int_0^T \int_{\dom} \kappa \frac{\Delta \rho^{\alpha/2+1}}{(\alpha/2+1)^2} \phi_m(\rho)  \nabla \rho^{\alpha/2+1}\cdot\varphi'(\widetilde{v}_m)\zeta(\rho)\theta dxdt \nonumber \\ 
  &- \int_0^T\int_{\dom} \sqrt{\frac{\rho}{\nu}} \tr(\T_{\nu}) \phi_m'(\rho) \rho v\cdot \varphi'(\widetilde{v}_m) \zeta(\rho)\theta dxdt =0. \nonumber
\end{align}


\subsection{Passing to the limit in the truncation of the velocity}

We now pass to the limit $m\to\infty$ in equation \eqref{eq: 4.3.rmomentum} by applying the dominated convergence theorem. It therefore remains to show the pointwise a.e.\ convergence of the integrands, since they are dominated by an integrable function. Since $\phi_m(y)\to 1$ and $\phi_m'(y)\to0$ as $m\to\infty$ for a.e.\ $y>0$, and since $\rho>0$ a.e.\ in $\Omega_T$ (which follows from
$\sigma \log\rho\in L^\infty(\Omega_T)$), we obtain
\begin{align*}
    \phi_m(\rho)\to1,\qquad
    \rho\phi_m'(\rho)\to0
    \quad\text{a.e. in }\Omega_T.
\end{align*}
Consequently, $\widetilde v_m=\phi_m(\rho)v\to v $ a.e.\ in $\Omega_T$ and $g(\widetilde v_m)\to g(v)$ a.e.\ in $\Omega_T$ for every $g\in W^{1,\infty}(\mathbb{R}^3)$. The pointwise a.e.\ convergence is straightforward for all terms except those involving $\nabla \widetilde{v}_m$. By rewriting 
\begin{align*}
  \sqrt\rho\na\widetilde{v}_m
  &= \sqrt\rho\na(\phi_m(\rho)v)
  = \sqrt\rho\phi'_m(\rho)\na\rho\otimes v
  + \sqrt\rho\phi_m(\rho) \na v \\
  &= 4\rho^{3/4}\phi'_m(\rho)\na\rho^{1/4}\otimes(\sqrt\rho v)
  + \phi_m(\rho)\rho^{-1/2}(\na(\rho v)-\na\rho\otimes v) \\
  &= 4\rho^{3/4}\phi'_m(\rho)\na\rho^{1/4}\otimes(\sqrt\rho v)
  + \phi_m(\rho)\nu^{-1/2}\mathbb{T}_\nu,
\end{align*}
we see that 
\begin{align*}
  \sqrt\rho\na\widetilde{v}_m\to \nu^{-1/2}\mathbb{T}_\nu
  \quad\mbox{a.e. in }\Omega_T.
\end{align*}
Hence, in the limit, we obtain the momentum equation \eqref{eq: 1.rmomentum} for the regularized system with the measures $R_1:=R_1^1+\cdots+R_1^6$ and $R_2:=R_2^1+\cdots+R_2^4$, where
\begin{align*}
  R_1^1 &= \rho \varphi(v)\zeta'(\rho) \partial_t \rho, &
  R_1^2 &= \rho\varphi(v) \zeta'(\rho)v\cdot \nabla \rho, \\
  R_1^3 &= - \sqrt{\nu\rho}\S_{\nu}:(\varphi'(v)\otimes\nabla\rho)
    \zeta'(\rho), & 
  R_1^4 &= - \sqrt{\nu} \zeta(\rho)\S_{\nu} : \bigg(\varphi''(v)\frac{\T_{\nu}}{\sqrt\nu}\bigg), \\
  R_1^5 &= - \kappa \rho^{\alpha/2+1} \frac{\Delta \rho^{\alpha/2+1}}{\alpha/2+1}\zeta'(\rho)\varphi'(v) \cdot \nabla \rho, &
  R_1^6 &= - \kappa \rho^{\alpha/2+1/2}\frac{\Delta \rho^{\alpha/2+1}}{\alpha/2+1}\varphi''(v): \frac{\T_{\nu}}{\sqrt\nu}\zeta(\rho), \\
  R_2^1 &= - \sqrt{\rho} \nabla^{2}\sqrt{\rho}: (\varphi'(v)\otimes \nabla \rho)\zeta'(\rho), &
  R_2^2 &= 4\sqrt{\rho}(\nabla \sqrt[4]{\rho} \otimes\nabla\sqrt[4]{\rho}): (\varphi'(v)\otimes \nabla \rho)\zeta'(\rho), \\
  R_2^3 &= - \nabla^2 \sqrt{\rho} : \bigg(\varphi''(v)\frac{\T_{\nu}}{\sqrt\nu}\bigg)\zeta(\rho), &
  R_2^4 &= 4 (\nabla \sqrt[4]{\rho}\otimes \nabla \sqrt[4]{\rho}) : \bigg(\varphi''(v)\frac{\T_{\nu}}{\sqrt\nu}\bigg)
  \zeta(\rho).
\end{align*}


\subsection{Establishing the viscosity relation}

The viscosity relation \eqref{eq: 1.rTnu} is rewritten following the same strategy as in \cite{LaVa18}. We therefore omit the proof and refer the reader to \cite[Sec.~4.1]{JuPh26} for the details. Here, the tensor $Q_{ijk}$ is defined as
\begin{align*}
   \int_0^T\int_\dom Q_{ijk}\phi dxdt = -\nu\sum_{\ell=1}^3\int_0^T\int_\dom
  v_k\frac{\pa^2\varphi}{\pa v_i\pa v_\ell}(v)\sqrt{\frac{\rho}{\nu}}
  (\T_\nu)_{j\ell}\phi dxdt.
\end{align*}


\subsection{Bounding the measures}

It remains to establish bounds as in \eqref{eq: 1.boundMeasure} for $R_1$, $R_2$, and $Q_{ijk}$. Since these measures are absolutely continuous with respect to the Lebesgue measure, it is sufficient to establish the corresponding $L^1(\Omega_T)$ bounds for their densities. Using Lemma~\ref{lem.reg} as well as directly the regularity \eqref{eq: 1.regularity} and \eqref{1.regul} of weak solutions, we estimate the terms $R_1^{i}$ as follows (recall that $\S_\nu=\frac12(\T_\nu+\T_\nu^T)$):
\begin{align*}
    \|R_1^1\|_{L^1(\Omega_T)}&\leq C \|\partial_t\rho\|_{L^{5/4}(\dom_T)} \|\rho^{\alpha/2+1} \|_{L^{10}(\dom_T)}  \|\rho^{-\alpha/2}\zeta'(\rho)\|_{L^{\infty}(\Omega_T)}\|\varphi(v)\|_{L^{\infty}(\Omega_T)} \\
    &\leq C  \|\rho^{-\alpha/2}\zeta'(\rho)\|_{L^{\infty}(\Omega_T)}\|\varphi(v)\|_{L^{\infty}(\Omega_T)},\\
    \|R_1^2\|_{L^1(\Omega_T)}&\leq\|\rho v\|_{L^2(\Omega_T)}\|\nabla \rho\|_{L^2(\Omega_T)}\| \varphi(v) \zeta'(\rho)\|_{L^{\infty}(\Omega_T)} \\
    &\leq C\| \varphi(v) \|_{L^{\infty}(\Omega_T)}\|\zeta'(\rho)\|_{L^{\infty}(\Omega_T)}, \\
    \|R_1^3\|_{L^1(\Omega_T)} &\leq C\|\sqrt{\rho}\|_{L^{10}(\Omega_T)}\|\S_{\nu}\|_{L^{2}(\Omega_T)}\|\nabla \rho\|_{L^{5/2}(\Omega_T)}\|\varphi'(v)\zeta'(\rho)\|_{L^\infty(\Omega_T)}\\
    &\leq C\|\varphi'(v)\|_{L^\infty(\Omega_T)} \|\zeta'(\rho)\|_{L^\infty(\Omega_T)}, \\
    \|R_1^4\|_{L^1(\Omega_T)}&\leq C \|\S_{\nu}\|_{L^2(\Omega_T)}^2 \|\zeta(\rho)\varphi''(v)\|_{L^{\infty}(\Omega_T)}
    \leq C \|\zeta(\rho)\|_{L^{\infty}(\Omega_T)}\|\varphi''(v)\|_{L^{\infty}(\Omega_T)}, \\
    \|R_1^5\|_{L^1(\Omega_T)}&\leq C \|\rho^{\alpha/2+1}\|_{L^{10}(\Omega_T)} \|\Delta \rho^{\alpha/2+1}\|_{L^2(\Omega_T)}\|\varphi'(v)\zeta'(\rho)\|_{L^{\infty}(\Omega_T)}\\
    &\leq C \|\varphi'(v)\|_{L^{\infty}(\Omega_T)}\|\zeta'(\rho)\|_{L^{\infty}(\Omega_T)}, \\
    \|R_1^6\|_{L^1(\Omega_T)}&\leq C \|\Delta \rho^{\alpha/2+1} \|_{L^2(\Omega_T)}\|\T_{\nu}\|_{L^2(\Omega_T)} \|\varphi''(v)\|_{L^{\infty}(\Omega_T)}\|\rho^{\alpha/2+1/2} \zeta(\rho)\|_{L^{\infty}(\Omega_T)}\\
    &\leq C\|\varphi''(v)\|_{L^{\infty}(\Omega_T)}\|\rho^{\alpha/2+1/2} \zeta(\rho)\|_{L^{\infty}(\Omega_T)}.
\end{align*}
The same strategy yields the following estimates for $R_2^{i}$:
\begin{align*}
     \sigma\|R_2^1\|_{L^{1}(\Omega_T)} &\leq \sqrt{\sigma}\|\sqrt\sigma \nabla^2 \sqrt{\rho}\|_{L^2(\Omega_T)}\|\sqrt{\rho} \zeta'(\rho)\|_{L^{\infty}(\Omega_T)} \|\nabla \rho\|_{L^2(\Omega_T)} \|\varphi'(v)\|_{L^{\infty}(\Omega_T)}\\
     &\leq C\sqrt{\sigma}\|\sqrt{\rho} \zeta'(\rho)\|_{L^{\infty}(\Omega_T)} \|\varphi'(v)\|_{L^{\infty}(\Omega_T)},\\
     \sigma\|R_2^2\|_{L^1(\Omega_T)}
     &\leq C \sqrt{\sigma} \|\sqrt{\rho} \zeta'(\rho)\|_{L^{\infty}(\Omega_T)} \|\sqrt[4]{\sigma} \nabla \sqrt[4]{\rho}\|_{L^4(\Omega_T)}^2 \|\varphi'(v)\|_{L^{\infty}(\Omega_T)} \|\nabla \rho\|_{L^2(\Omega_T)}\\
     &\leq C\sqrt{\sigma}\|\sqrt{\rho} \zeta'(\rho)\|_{L^{\infty}(\Omega_T)} \|\varphi'(v)\|_{L^{\infty}(\Omega_T)},\\
     \sigma\|R_2^3\|_{L^1(\Omega_T)}&\leq \sqrt{\sigma}\nu^{-1/2} \|\sqrt{\sigma}\nabla^2 \sqrt{\rho}\|_{L^2(\Omega_T)} \|\T_{\nu}\|_{L^2(\Omega_T)}\|\varphi''(v)\zeta(\rho)\|_{L^{\infty}(\Omega_T)}\\
     &\leq C\sqrt{\sigma} \|\zeta(\rho)\|_{L^{\infty}(\Omega_T)} \|\varphi''(v)\|_{L^{\infty}(\Omega_T)},\\
     \sigma\|R_2^4\|_{L^1(\Omega_T)} &\leq C \sqrt{\sigma} \|\sqrt[4]{\sigma}\nabla \sqrt[4]{\rho}\|_{L^4(\Omega_T)}^2 \|\T_{\nu}\|_{L^2(\Omega_T)} \|\varphi''(v)\zeta(\rho)\|_{L^{\infty}(\Omega_T)}\\
     &\leq C\sqrt{\sigma} \|\zeta(\rho)\|_{L^{\infty}(\Omega_T)} \|\varphi''(v)\|_{L^{\infty}(\Omega_T)}.
\end{align*}
A similar estimate holds for $Q_{ijk}$. These estimates prove the bound \eqref{eq: 1.boundMeasure} in Definition \ref{def.renormregul}.


\section{From renormalized weak to weak solutions}\label{sec.rTOwk}

Let $(\sqrt{\rho}, \sqrt{\rho}v)$ be a renormalized weak solution to \eqref{eq: 1a.mass}--\eqref{eq: 1a.momentum} with $\sigma\ge0$ on $[0,T]$ and initial conditions $(\sqrt{\rho^0}, \sqrt{\rho^0} v^0)$ such that \eqref{eq: 1.regIC} holds and $\sqrt{\rho^0}$ is uniformly bounded away from zero if $\sigma>0$. To recover the weak formulation from the renormalized formulation, one would formally choose $\varphi(y)=y_i$ for $i=1,2,3$ and $\zeta(z)=1$. Since these functions do not satisfy the admissibility assumptions, we approximate them by suitable functions and then pass to the limit. To this end, we define a smooth function $\Phi$ satisfying
\begin{align*}
    &\Phi(z)= 1 \mbox{ for } z \in [-1,1],\quad \Phi(z)=0 \mbox{ for } z\notin [-2,2],\quad \Phi(z) \in [0,1].
\end{align*}
We further define for $\delta>0$ and $i=1,2,3$:
\begin{align*}
    \varphi_{\delta}^{i} (y) = \frac{1}{\delta} \Psi(\delta y_i) \prod_{j=1}^3 \Phi(\delta y_j),\quad\mbox{where }\Psi(z)= \int_0^z \Phi(s)ds.
\end{align*}
It holds by construction that $\varphi_{\delta}^{i} \in W^{3,\infty}(\mathbb{R}^3)$ and, since this function is compactly supported, it satisfies inequality \eqref{eq: 1.varphi}. Moreover, we have
\begin{align*}
    & \|\varphi_{\delta}^{i}\|_{L^{\infty}}\leq \frac{C}{\delta}, \quad \|(\varphi_{\delta}^{i})'\|_{L^{\infty}}\leq C,  \quad \|(\varphi^{i}_{\delta})''\|_{L^{\infty}}\leq  C\delta, \\
    & \varphi_{\delta}^{i} (y)\rightarrow y_i, \quad (\nabla_y\varphi_{\delta}^{i}) (y)\rightarrow e_i
    \quad\mbox{as }\delta\to 0,
\end{align*}
where $e_i$ is the $i$th unit vector of $\R^3$. Furthermore, we define the function $\zeta_\lambda(z) = \Phi(\lambda z)$ for $\lambda>0$, which converges pointwise to one as $\lambda\to 0$ and satisfies for $0<a<1$:
\begin{align*}
    \|\zeta_{\lambda}\|_{L^{\infty}} \leq 1, \quad |z|^{a}\zeta_{\lambda}(z)\leq \frac{C}{\lambda^{a}}, \quad |z|^{a} |\zeta'_{\lambda}(z)| \leq C \lambda^{1-a}.
\end{align*}
Choosing $\delta= \lambda^{\beta}$ with $\beta= 3/4+\alpha/2$, inequality \eqref{eq: 1.boundMeasure} implies that
\begin{align*}
    \|R_1\|_{\mathcal{M}(\Omega_T)}
    + \sigma\|R_2\|_{\mathcal{M}(\Omega_T)}
    +\sum_{i,j,k=1}^3\| Q_{ijk}\|_{\mathcal{M}(\dom_T)} \rightarrow 0
\end{align*}
as $\lambda \rightarrow 0$ and hence 
\begin{align*}
    \int_0^T\int_{\dom} (R_1 + \sigma R_2) \theta dxdt\rightarrow 0,\quad \int_0^T\int_{\dom} Q_{ijk}\theta dxdt\rightarrow 0.
\end{align*}
The remaining terms in the momentum equation and viscosity relation converge by dominated convergence.


\section{Limit $\sigma \rightarrow 0$}\label{sec.sigma}

Consider a sequence $\sigma\to0$ and a sequence of initial conditions $(\sqrt{\rho_{\sigma}^0},\sqrt{\rho_{\sigma}^0}v_{\sigma}^0)\rightarrow (\sqrt{\rho^0},\sqrt{\rho^0}v^0)$ converging in $(H^1(\dom)\cap L^{2\gamma}(\dom))\times L^2(\dom;\R^3)$ such that \eqref{eq: 1.regIC} is uniformly satisfied and
\begin{align*}
    \sigma\bigg|\int_{\T^3}\log\rho^0_{\sigma}(x) dx\bigg|\leq C
\end{align*}
holds. Let $(\sqrt{\rho_\sigma}, \sqrt{\rho_\sigma}v_\sigma)$ be a sequence of corresponding renormalized weak solutions to \eqref{eq: 1a.mass}--\eqref{eq: 1a.momentum} on $[0,T]$. In this section, we show that there exists a subsequence of renormalized weak solutions converging to a weak solution to \eqref{eq: 1a.mass}--\eqref{eq: 1a.momentum} with $\sigma=0$.

The following compactness arguments rely on the estimates in \eqref{eq: 1.regularity}, whose constants depend only on the initial data and $T$, and are therefore uniform with respect to $\sigma$. We also repeatedly use the estimates derived from these bounds, which are collected in Lemma~\ref{lem.reg}. The strong convergence of $\rho$ follows from the Aubin--Lions lemma:
\begin{align*}
    \rho_\sigma \rightarrow \rho \quad \mbox{strongly in } L^{p}(\Omega_T) \mbox{ as }\sigma\to 0\mbox{ for all } 
    p<q_1 = 5(\alpha+2),
\end{align*}
and therefore, for any $\zeta\in L^{\infty}(\R)$,
\begin{align}
    \zeta(\rho_{\sigma})\rightarrow\zeta(\rho) \quad \mbox{strongly in } L^q(\Omega_T)\mbox{ for all } q<\infty.\label{eq: zetaConv}
\end{align}
It is possible to show that $\partial_t (\rho_{\sigma}v_{\sigma}) \in L^q(0,T;H^{-s}(\dom))$ for some $q>1$, $s\in \N$, using the fact that $\diver\K(\rho_{\sigma})$ is uniformly bounded in $L^{10/9}(0,T; W^{1,3}(\dom))$; see \eqref{3.K}. Hence, the Aubin--Lions lemma yields
\begin{align}\label{7.m}
    \rho_{\sigma} v_\sigma \rightarrow  m \quad \text{strongly in } L^{p}(\Omega_T) \text{ for all } p<2.
\end{align}
We deduce from the proof of Lemma~\ref{lem.strConvGradient} that
\begin{align*}
    \nabla \rho_\sigma^{\alpha/2+1} \rightarrow\nabla\rho^{\alpha/2+1}  \quad \mbox{strongly in } L^p(\Omega_T) \ \text{for all } p<10/3.
\end{align*}
By the Banach--Alaoglu theorem, we obtain the following convergences (up to subsequences):
\begin{align*}
   \T_{\nu, \sigma} \rightharpoonup \T_{\nu}&\quad \mbox{weakly in } L^2(\Omega_T)\\
    \sqrt{\rho_{\sigma}} v_{\sigma} \rightharpoonup\Lambda &\quad \mbox{weakly* in } L^{\infty}(0,T; L^2(\dom)),\\
    \nabla \rho^{\gamma/2}_{\sigma}\rightharpoonup \nabla \rho^{\gamma/2} &\quad \mbox{weakly in } L^2(\Omega_T),\\
    \nabla \rho_\sigma^{\alpha/4+1/2} \rightharpoonup\nabla \rho ^{\alpha/4+1/2} &\quad \mbox{weakly in }L^4(\Omega_T),\\
    \Delta \rho_{\sigma}^{\alpha/2+1} \rightharpoonup\Delta \rho^{\alpha/2+1}&\quad \mbox{weakly in }L^2(\Omega_T).  
\end{align*}
The limit in $\rho_\sigma v_\sigma = \sqrt{\rho_\sigma} \sqrt{\rho_\sigma}v_\sigma$ yields $\sqrt{\rho} \Lambda =m$ a.e. We wish to identify the weak limit $m$, introduced in \eqref{7.m}. For this, we define the velocity $v$ by
\begin{align*}
    v(t,x) = \begin{cases}\displaystyle
        \frac{m(t,x)}{\rho(t,x)} &\mbox{for }(t,x)\in\{\rho>0\},\\
        0 &\mbox{for }(t,x)\in\{\rho=0\}.
    \end{cases}
\end{align*}
It follows from Fatou's Lemma that $m=\rho v = \sqrt{\rho} \Lambda$ a.e. The following convergences are proved in \cite[Lem.~4.2]{AnSp22} for continuous and bounded functions $H$:
\begin{align}
    \rho_{\sigma}^{\beta} H(v_{\sigma}) \rightarrow \rho^{\beta} H(v) &\quad \mbox{strongly in }L^p(\Omega_T),\  p<5(\alpha+2)/\beta,\label{eq: rhoH(v)}\\
    \rho_\sigma^{\gamma/2}H(v_{\sigma}) \rightarrow \rho^{\gamma/2} H(v) &\quad \mbox{strongly in } L^p(\Omega_T),\ p < 10/3,\label{eq: rhoGammaH(v)}\\
    \nabla \rho_{\sigma}^{\alpha/2+1} H(v_\sigma)\rightarrow \nabla\rho^{\alpha/2+1} H(v) &\quad \mbox{strongly in } L^p(\Omega_T),\ p<10/3.\label{eq: gradientH(v)}
\end{align}

The passage to the limit in the viscosity relation is identical to that in \cite{LaVa18}; for details, we refer to \cite{JuPh26}. The momentum equation is treated term by term. For the Korteweg term, we set $p=3$. The desired convergence then follows from \eqref{eq: rhoH(v)} with $\beta= \alpha/2+1$ and \eqref{eq: gradientH(v)}, combined with \eqref{eq: zetaConv} with $q=6$:
\begin{align*}
        \int_{0}^T\int_{\dom} &\rho^{\alpha/2+1}_{\sigma}\frac{\Delta \rho_{\sigma}^{\alpha/2+1}}{\alpha/2+1}\varphi'(v_{\sigma})\cdot \nabla \theta \zeta(\rho_{\sigma})dxdt \\
        &\rightarrow \int_{0}^T\int_{\dom} \rho^{\alpha/2+1}\frac{\Delta \rho^{\alpha/2+1}}{\alpha/2+1}\varphi'(v)\cdot\nabla \theta \zeta(\rho)dxdt, \\
        \int_{0}^T\int_{\dom} &\frac{\Delta \rho_{\sigma}^{\alpha/2+1}}{(\alpha/2+1)^2}\nabla\rho^{\alpha/2+1}_{\sigma}\cdot\varphi'(v_{\sigma})\zeta(\rho_{\sigma}) \theta dxdt \\
        &\rightarrow \int_{0}^T\int_{\dom} \frac{\Delta \rho^{\alpha/2+1}}{(\alpha/2+1)^2}\nabla\rho^{\alpha/2+1}\cdot\varphi'(v)\zeta(\rho) \theta dxdt.
\end{align*}
Likewise, the convergence in the convective terms follows directly from \eqref{eq: rhoH(v)} with $p=2$ and \eqref{eq: zetaConv} with $q=2$, since the mappings $y\mapsto \varphi(y)$ and $y\mapsto y\varphi(y)$ are bounded and continuous. The convergence in the viscosity term follows from \eqref{eq: rhoH(v)} with $\beta=1/2$ and $p=5$, together with \eqref{eq: zetaConv} with $q=10/3$, while the convergence in the pressure term is obtained from \eqref{eq: rhoGammaH(v)} with $p=3$ and \eqref{eq: zetaConv} with $q=3$. Finally, the remaining terms converge to zero, due to the following estimations:
\begin{align*}
  \|\sigma &\rho_{\sigma} |v_{\sigma}|^2 \zeta(\rho_{\sigma}) \varphi'(v_{\sigma})\|_{L^1(\Omega_T)} \\
  &\leq \sqrt[4]{\sigma}\|\sqrt[4]{\rho_{\sigma}}\|_{L^4(\Omega_T)}
  \|\sqrt[4]{\sigma}\sqrt[4]{\rho_{\sigma}}v_{\sigma}\|_{L^4(\Omega_T)}^3\|\zeta(\rho_{\sigma})\varphi'(v_{\sigma})\|_{L^{\infty}(\Omega_T)}
  \leq \sqrt[4]{\sigma} C, \\
  \|\sigma &v_{\sigma}\varphi'(v_{\sigma})\zeta(\rho_{\sigma})
  \|_{L^{1}(\Omega_T)} \\
  &\leq  C\sqrt{\sigma} \|\sqrt{\sigma}v_{\sigma}\|_{L^2(\dom)} \|\zeta(\rho_{\sigma})\varphi'(v_{\sigma})\|_{L^{\infty}(\Omega_T)}
  \leq \sqrt{\sigma} C, \\
  \|\sigma& \sqrt{\rho_\sigma} \nabla^2 \sqrt{\rho_\sigma} 
  \varphi'(v_\sigma)\zeta(\rho_\sigma)\|_{L^1(\Omega_T)} \\
  &\leq \sqrt{\sigma} \|\sqrt{\sigma}\nabla^2\sqrt{\rho_\sigma}\|_{L^2(\Omega_T)}\| \sqrt{\rho_\sigma} \zeta(\rho_{\sigma})\|_{L^\infty(\Omega_T)}
  \|\varphi'(v_\sigma)\|_{L^\infty(\Omega_T)}
  \leq \sqrt{\sigma}C, \\
  \|4\sigma& \sqrt{\rho}\nabla \sqrt[4]{\rho_{\sigma} \otimes 
  \nabla \sqrt[4]{\rho_{\sigma}} }\varphi'(v_\sigma)\zeta(\rho_\sigma)\|_{L^1(\Omega_T)} \\
  &\leq  C\sqrt{\sigma} \|\sqrt[4]{\sigma}\nabla\sqrt[4]{\rho_{\sigma}}
  \|^2_{L^4(\Omega_T)}\|\sqrt{\rho_{\sigma}}\zeta(\rho_{\sigma})
  \|_{L^{\infty}(\Omega_T)} \|\varphi'(v_{\sigma})\|_{L^{\infty}(\Omega_T)}
  \leq \sqrt{\sigma} C.
\end{align*}
This completes the proof of the limit $\sigma\to0$ and, consequently, the proof of Theorem~\ref{theorem.ex}.

    
\section{Remarks}\label{sec.rem}

\begin{remark}[Range for $\alpha$]\label{rem.alpha}\rm 
An inspection of the proof shows that the assumption $\alpha\geq -1$ is used essentially in Section~\ref{sec.wkTOr}. Indeed, in the limit $m\to\infty$, we can establish only the convergence $\sqrt{\rho}\nabla v_m\rightarrow \nu^{-1/2} \T_{\nu}$, whereas convergence of $\nabla v_m$ itself is not available. Consequently,
\begin{align*}
    \int_0^T\int_{\dom}& \kappa \rho^{\alpha/2+1} \frac{\Delta \rho^{\alpha/2+1}}{\alpha/2+1} \phi_m(\rho) \varphi''(\widetilde{v}_m) : \nabla\widetilde{v}_m \zeta(\rho)\theta dxdt  \\
    &\rightarrow \int_0^T\int_{\dom} \kappa \rho^{\alpha/2+1/2} \frac{\Delta \rho^{\alpha/2+1}}{\alpha/2+1} \varphi''(v) : \frac{\T_{\nu}}{\sqrt\nu} \zeta(\rho)\theta dxdt,
\end{align*}
and the limit term contributes to $R_1$. Estimating this contribution requires sufficient regularity of $\rho^{\alpha/2+1/2}$, which is available only for $\alpha\geq -1$. Although the assumption $\alpha\geq -1$ is invoked at several earlier stages of the proof, those arguments could likely be adapted to a broader range of $\alpha$. The above step appears to be the only place at which this restriction is genuinely essential.
\end{remark}

\begin{remark}[General viscosity functions]\label{rem.visc}\rm 
    One may ask whether the approach extends to viscosity coefficients of the form $\mu_1(\rho)=\rho^\beta$ with $\beta>0$. A first ingredient is a generalization of the auxiliary inequality from Section~\ref{sec.ineq},
    \begin{align*}
        -\int_{\mathbb{T}^d}\na\rho^\beta\cdot\na\bigg(
        \rho^{\alpha/2}\frac{\Delta\rho^{\alpha/2+1}}{\alpha/2+1}\bigg)dx
        \ge C_0^*\int_{\mathbb{T}^d}\big(|\na\rho^{\theta/2}|^4
         + |\Delta\rho^\theta|^2\big)dx, 
    \end{align*}
    where $\theta=(\alpha+\beta+1)/2$ and $C_0^*>0$, for appropriate ranges of $\alpha$ and $\beta$. Such an inequality can be proved using the same method as in Theorem \ref{theorem.ineq}. One could then consider an approximation scheme incorporating the Korteweg term associated with the generalized viscosity together with additional regularization terms, following the strategy of \cite{BVY22}. In that work, the authors establish an existence theory for the compressible Navier--Stokes system with generalized viscosity coefficients. Their approximation scheme introduces drag terms, an additional pressure, and an approximation of the viscosity coefficient, and the limit passage relies on the framework of renormalized weak solutions. To solve our equations, one needs to justify the contribution of the Korteweg term, which is independent of the viscosity term. It is likely that, as in \cite{AnSp22} and also in our work, the notion of renormalized weak solutions must be modified by introducing a truncation of the density. Since each step of the compactness argument depends sensitively on this structure, such an extension requires a careful reexamination of the entire approximation procedure and is therefore left for future work.
\end{remark}

\begin{remark}[Inequality \eqref{2.ineq} in one dimension]
\label{rem.1D}\rm
We can derive the optimal range $-2<\alpha<1$ for \eqref{2.ineq} in one space dimension by systematic integration by parts. All possible integrations by parts are encoded in the ``dummy'' integrals
\begin{align*}
  J_1 &= \int_{\T}\bigg(
  \lambda^2\bigg(\frac{\lambda_x}{\lambda}\bigg)^3\bigg)_x dx
  = \int_{\T}\lambda^2(-\xi_1^4+3\xi_1^2\xi_2)dx = 0, \\
  J_2 &= \int_{\T}\bigg(\lambda^2\frac{\lambda_x}{\lambda}
  \frac{\lambda_{xx}}{\lambda}\bigg)_x dx
  = \int_{\T}\lambda^2(\xi_2^2 + \xi_1\xi_3)dx = 0, \\
  J_3 &= \int_{\T}\bigg(\lambda^2\frac{\lambda_{xxx}}{\lambda}
  \bigg)_x dx = \int_{\T}\lambda^2(\xi_1\xi_3 + \xi_4)dx = 0,
\end{align*}
where $\xi_j=(\pa_x^j\lambda)/\lambda$ for $j=1,2,3,4$. With the notation of the proof of Theorem \ref{theorem.ineq}, we have in one space dimension
\begin{align*}
  L = \int_{\T}\lambda^2\bigg(\frac{1}{(\beta+1)^2}\xi_2^2 
  - \frac{\beta}{(\beta+1)^3}\xi_1^2\xi_2\bigg)dx, \quad
  R = \int_{\T}\lambda^2\bigg(\frac{1}{16}\xi_1^4 + \xi_2^2\bigg)dx,
\end{align*}
recalling that $\beta=\alpha/2$. Inequality \eqref{2.in} is shown if we can find $C_0>0$, $C_1$, $C_2$, $C_3\in\R$ such that
\begin{align*}
  L - {}& C_0 R + C_1J_1 + C_2J_2 + C_3J_3 \\
  &= \int_{\T}\frac{\lambda^2}{(\beta+1)^3}\bigg\{
  (\beta + 1 - K_0 + K_2)\xi_2^2 + (3K_1 - \beta)\xi_1^2\xi_2
  - \bigg(\frac{K_0}{16} + K_1\bigg)\xi_1^4 \\
  &+ (\beta+1)^3(K_2+K_3)\xi_1\xi_3 + (\beta+1)^3K_3\xi_4
  \bigg\}dx,
\end{align*}
where $K_i=C_i/(\beta+1)^3$ for $i=0,1,2,3$. The integrand contains $\xi_4$ in first power, so it cannot have a sign and $K_3$ must vanish. Also the variable $\xi_3$ appears in first power, and we need to choose $K_2=0$ to eliminate the term $\xi_1\xi_3$. This means that only one integration by parts is relevant, defined by $J_1$. We wish to find $K_0>0$ and $K_1\in\R$ such that for all $(\xi_1,\xi_2)\in\R^2$,
\begin{align*}
  (\beta + 1 - K_0)\xi_2^2 + (3K_1 - \beta)\xi_1^2\xi_2 
  - \bigg(\frac{K_0}{16} + K_1\bigg)\xi_1^4 \ge 0.
\end{align*}
This is the case if and only if $K_0<\beta+1$ and
\begin{align*}
  0 &\le -4(\beta+1-K_0)\bigg(\frac{K_0}{16} + K_1\bigg) 
  - (3K_1 - \beta)^2  \\
  &= -9\bigg(K_1 - \frac19(\beta-2+2K_0)\bigg)^2 - \beta^2 
  - \frac14(\beta+1-K_0)K_0 + \frac19(\beta-2+2K_0)^2.
\end{align*}
Choosing the maximizing value $K_1=(\beta-2+2K_0)/9$, we need to find $0<K_0<\beta+1$ such that 
\begin{align*}
  0 \le f(K_0) := -\beta^2 - \frac14(\beta+1-K_0)K_0 
  + \frac19(\beta-2+2K_0)^2.
\end{align*}
If $f(0)>0$, we can find, by continuity, a value $K_0>0$ such that $f(K_0)\ge 0$. Now, $f(0)>0$ is equivalent to $-\beta^2+(\beta-2)^2/9 = -4(\beta+1)(2\beta-1)/9 > 0$. This yields the conditions $\beta+1>0$, $2\beta-1<0$, which is equivalent to $-2<\alpha<1$. 
\end{remark}


\end{document}